\documentclass[12pt,reqno]{amsart}
\usepackage[a4paper,top=3cm,bottom=3cm,inner=3cm,outer=3cm]{geometry}

\usepackage{amsmath, amssymb, amsfonts, mathtools}
\usepackage{wasysym}

\usepackage{xcolor}
\definecolor{darkgreen}{rgb}{0,0.45,0}
\usepackage[
    colorlinks, 
    citecolor=blue, 
    urlcolor=darkgreen, 
    final, 
    hyperindex, 
    linkcolor = blue
]{hyperref}

\usepackage[capitalize]{cleveref}
\usepackage{xspace}
\usepackage{youngtab}
\usepackage{stackrel}

\usepackage{amsthm}
\usepackage{enumerate}

\theoremstyle{plain}
\newtheorem{thm}{Theorem}
\newtheorem*{thm*}{Theorem}
\newtheorem{prop}{Proposition}
\newtheorem{lem}{Lemma}
\newtheorem*{prop*}{Proposition}
\newtheorem{cor}{Corollary}

\theoremstyle{definition}
\newtheorem{defn}{Definition}
\newtheorem{example}{Example}

\theoremstyle{remark}
\newtheorem{rmk}{Remark}

\theoremstyle{plain} % just in case the style had changed
\newcommand{\thistheoremname}{}
\newtheorem*{genericthm*}{\thistheoremname}
\newenvironment{namedthm*}[1]
  {\renewcommand{\thistheoremname}{#1}%
   \begin{genericthm*}}
  {\end{genericthm*}}

\title[]{Categorical Lyapunov Theory I: 
\\Stability of Flows
}

\date{\today}
\author[Ames]{Aaron D. Ames}
\author[Moeller]{Joe Moeller}
\author[Tabuada]{Paulo Tabuada}

\address{California Institute of Technology, Pasadena, CA, USA}
\address{University of California, Los Angeles, Los Angeles, CA, USA}

\email{ames@caltech.edu,
jmoeller@caltech.edu,
tabuada@ee.ucla.edu}

\newcommand{\flow}{\phi}
\newcommand{\F}{\mathcal F}

\newcommand{\D}{\mathrm D}

\newcommand{\ex}[1]{\hspace*{\fill}\textcolor{gray}{#1}}
\newcommand{\exmath}[1]{\tag*{\textcolor{gray}{$#1$}}}
\newcommand{\excenter}[1]{\textcolor{gray}{#1}}

\newcommand{\R}{\mathbb R}
\newcommand{\Rplus}{\R_{\geq 0}}
\newcommand{\Rplusinfty}{\Rplus^{\infty}}

\newcommand{\Z}{\mathbb Z}
\newcommand{\Zplus}{{\Z_{\geq 0}}}

\newcommand{\Q}{\mathcal Q}
\newcommand{\V}{\mathcal V}
\newcommand{\VCat}{\V\Cat}

\renewcommand{\P}{\category P}

\newcommand{\id}{\mathrm{id}}
\newcommand{\inv}{^{-1}}

\newcommand{\norm}{\|\cdot\|_{x^*}}
\newcommand{\xnorm}[1]{\|#1\|_{x^*}}

\newcommand{\define}[1]{{\bf \boldmath{#1}}}

\newcommand{\maps}{\colon}
\newcommand{\To}{\Rightarrow}

\newcommand{\Ob}{\mathrm{Ob}}
\newcommand{\op}{^\mathrm{op}}
\newcommand{\disc}{\mathsf{disc}}

\newcommand{\T}{\mathrm{T}}

\newcommand{\K}{\mathcal K}

\newcommand{\category}[1]{\mathsf{#1}}
\newcommand{\C}{\category C}
\newcommand{\E}{\category E}

\newcommand{\namedcat}[1]{\mathsf{#1}}

\newcommand{\TFlow}{T\mhyphen\Flow}

\newcommand{\Cat}{\namedcat{Cat}}
\newcommand{\LMet}{\namedcat{LMet}}
\newcommand{\Man}{\namedcat{Man}}
\newcommand{\Set}{\namedcat{Set}}

\newcommand{\Flow}{\namedcat{Flow}}
\newcommand{\Top}{\namedcat{Top}}
\newcommand{\VCatgen}{\VCat_{L}}

\mathchardef\mhyphen="2D

\usepackage{tikz, tikz-cd}
\usetikzlibrary{positioning}
\definecolor {processblue}{cmyk}{0.9,0.5,0,0}

\tikzstyle{simple}=[-,line width=2.000]
\tikzstyle{arrow}=[-,postaction={decorate},decoration={markings,mark=at position .5 with {\arrow{>}}},line width=1.100]
\pgfdeclarelayer{edgelayer}
\pgfdeclarelayer{nodelayer}
\pgfdeclarelayer{bg}
\pgfsetlayers{bg,edgelayer,nodelayer,main}
\tikzstyle{none}=[inner sep=-1pt]
\tikzstyle{species}=[circle,fill=none,draw=black,scale=1.0]
\tikzstyle{transition}=[rectangle,fill=none,draw=black,scale=1.15]
\tikzstyle{empty}=[circle,fill=none, draw=none]
\tikzstyle{inputdot}=[circle,fill=black,draw=black, scale=.5]
\tikzstyle{dot}=[circle,fill=black,draw=black]
\tikzstyle{bounding}=[circle,dashed, fill=none,draw=black, scale=9.00]
\tikzstyle{simple}=[-,draw=black,line width=1.000]
\tikzstyle{inarrow}=[-,draw=black,postaction={decorate},decoration={markings,mark=at position .5 with {\arrow{>}}},line width=1.000]
\tikzstyle{tick}=[-,draw=black,postaction={decorate},decoration={markings,mark=at position .5 with {\draw (0,-0.1) -- (0,0.1);}},line width=1.000]
\tikzstyle{inputarrow}=[->,draw=black, shorten >=.05cm]

\tikzset{main node/.style={circle,fill=blue!20,draw,minimum size=1cm,inner sep=0pt},}

\tikzstyle{construct}=[fill=white, draw=black, shape=circle]
\tikzstyle{universal}=[fill=black, draw=black, shape=circle]

\begin{document}

\begin{abstract}
    Lyapunov's theorem provides a fundamental characterization of the stability of dynamical systems. This paper presents a categorical framework for Lyapunov theory, generalizing stability analysis with Lyapunov functions categorically. Core to our approach is the set of axioms underlying a \emph{setting for stability}, which give the necessary ingredients for ``doing Lyapunov theory'' in a category of interest. With these minimal assumptions, we define the stability of equilibria, formulate \emph{Lyapunov morphisms}, and demonstrate that the existence of Lyapunov morphisms is necessary and sufficient for establishing the stability of flows. To illustrate these constructions, we show how classical notions of stability, e.g., for continuous and discrete time dynamical systems, are captured by this categorical framework for Lyapunov theory. Finally, to demonstrate the extensibility of our framework, we illustrate how enriched categories, e.g., Lawvere metric spaces, yield settings for stability enabling one to ``do Lyapunov theory'' in enriched categories. 
\end{abstract}

\maketitle
\setcounter{tocdepth}{1} % 1sections, 2subsections
\tableofcontents

\section{Introduction}
\label{sec:intro}

Lyapunov theory completely characterizes the stability of dynamical systems. These elegant conditions for establishing stability have proven to be the single most powerful method for checking the stability of nonlinear systems. For example, Lyapunov analysis is used to prove the stability of model predictive control (MPC) \cite{rawlings2017model}. It has lead to a wide-variety of nonlinear control methods, e.g., control Lyapunov functions (CLFs) for stabilization \cite{sontag1989universal}. More generally, there is a wide range of ``Lypaunov-like'' functions for establishing a broader range of properties, e.g., control barrier functions (CBFs) for safety (forward set invariance) \cite{ames2016control}. The prevalence of these conditions seems to imply a common framework for ``Lyapunov-like'' conditions establishing stability in an abstract sense. 

The goal of this work is to generalize Lyapunov-like constructions under a single universal framework. Inspired by the work of Quillen \cite{quillen2006homotopical}, which provided axioms on a category that allows one to ``do homotopy theory," we seek to axiomatize Lyapunov theory categorically. Category theory provides an ideal setting in which to find the commonality between superficially distinct logical frameworks. For example, \cite{haghverdi2005bisimulation} provides a categorical framework for bisimulation relations that enables their application to  disparate application domains. Inspired by these examples we formulate the notion of a \emph{setting for stability} that allows one to ``do Lyapunov theory'' in a suitable category, and demonstrate that this captures classical Lyapunov theory for stability. 

\subsection{Motivation: Classical Stability Revisited}
\label{sec:motivation}

We begin by revisiting classical stability results with a view toward motivating the categorical constructions that will be presented in the paper. 

Consider a dynamical system described by an ordinary differential equation: 
\[
\dot{x} = f(x) 
\]
where $f \maps E \subseteq \R^n \to \R^n$. A (forward complete) solution $c \maps \Rplus \to E$ satisfies the above ODE: 
\begin{eqnarray}
\label{eqn:odesolution}
    \dot{c}(t) = f(c(t)).
\end{eqnarray}
For the purpose of aligning the notions of systems theory with the language and practice of category theory, we can express the equation above as a commutative diagram:
\begin{eqnarray}
\label{eqn:solutiondiagram}
\begin{tikzcd}
    \Rplus\times \R 
    \ar[r, "Tc"]
    &
    E \times \R^n 
    \\
    \Rplus 
    \ar[r, "c"'] 
    \ar[u, "\vec{1} \coloneq (\id_{\Rplus}{,}1)"]
    &
    E   
    \ar[u, "\vec{f}\coloneq(\id_{E}{,}f)"']
    \ar[ul, phantom, "="]
\end{tikzcd}
\end{eqnarray}
where $\vec{1} = (\id_{\Rplus},1)$ is the vector field representing the unit clock: $\dot{t} = 1$ and $\vec{f}$ is the vector field associated with the ODE $\vec{f} = (\id_{E},f)$. The diagram commutes if: 
\[
Tc \circ \vec{1} = Tc(t,1) = (c(t) ,\dot{c}(t) 1) = (c(t),f(c(t))) = \vec{f} \circ c, 
\]
that is, if \eqref{eqn:odesolution} is satisfied. Flows $\flow \maps \Rplus \times E \to E$ are the collections of solutions parameterized by initial condition: $\flow_t(x_0) = c(t)$ for $c(0) = x_0$. 

We are fundamentally interested in studying the properties of equilibrium points, those points $x^* \in E$ which remain fixed under the dynamics of the system. An equilibrium point is \emph{stable} if a solution that starts near the point stays near the point for all time. Formally, for all $\epsilon > 0$, there exists a $\delta > 0$ such that for any solution curve $c \maps \Rplus \to E$
\begin{eqnarray}
\label{eqn:stabilitydef}
    \| c(0)-x^* \| \leq \delta \quad \Rightarrow \quad \|c(t)-x^*\| \leq \epsilon 
    \quad \forall t \in \Rplus.
\end{eqnarray}
Stability can be equivalently reformulated in a modern form: an equilibrium point $x^* \in E$ is stable if 
\begin{eqnarray}
\label{eqn:stabilityalphaintro}
    \| c(t) \|_{x^*}    \leq    \alpha(\|c(0)\|_{x^*}). 
\end{eqnarray}
Here $\| x \|_{x^*} \coloneq \| x - x^* \|$, and $\alpha \maps \Rplus \to \Rplus$ is a \emph{class $\K$ function}---continuous, strictly increasing, and $\alpha(0) = 0$. This implies $\alpha$ admits a partial inverse and we recover the classical notion of stability by picking $\delta = \alpha^{-1}(\epsilon)$. Importantly, the class $\K$ formulation of stability facilitates generalizations, e.g., to forward set invariance via barrier functions \cite{ames2016control}. We leverage this formulation of stability to express stability through diagrams that \emph{commute up to inequality}.
\begin{eqnarray}
\label{eqn:stabilityintro}
\begin{tikzcd}
    \R  
    \arrow[rr, "c"]
    \arrow[d, "c(0)"']
    &&
    E
    \arrow[d, "\|\cdot \|_{x^*}"]
    \\
     E
    \arrow[r, "\|\cdot\|_{x^*}"']
    \ar[urr, phantom, "\geq"]
     &
    \Rplus
    \arrow[r, "\alpha"']
    &
    \Rplus
\end{tikzcd}
\end{eqnarray}
Saying that the ``diagram commutes up to inequality'' simply implies that \eqref{eqn:stabilityalphaintro} is satisfied. For ``commuting up to inequality'' to be sensible categorically, one must formalize this notion in the context of lax commuting diagrams.

Lyapunov's method \cite{lyapunov1892general}\footnote{See \cite{lyapunov1992general} for a translation. See \cite{NonlinearSystems} for a modern treatment.} provides a complete characterization of stability through conditions that can be checked directly on the ODE. 
A \define{Lyapunov function} is a continuously differentiable function $V \maps E \to \R$, satisfying the following conditions for all $x \in E$: 
\begin{eqnarray}
\label{eqn:lyapunov_condition}
\textit{Positive Definite:} \qquad V(x) & \geq   & 0 \nonumber\\
\qquad V(x) & =   & 0  \quad  \mathrm{iff} \quad x = x^* \\
\textit{Decrescent:}  \qquad \dot{V}(x) & = &  \frac{\partial V}{\partial x}\Big \vert_{x} f(x)  \leq 0 . \nonumber
\end{eqnarray}
The existence of a Lyapunov function implies stability of the equilibrium point $x^*$.
Conversely, every system with a stable equilibrium point admits a Lyapunov function \cite{Massera} (cf. \cite{MattenetJungers} which establishes this using category theory). 

This paper builds on the key observation: \emph{Lyapunov functions can be viewed as a map from the dynamical system to the ``simplest'' stable system} \cite{ames2006stability}. We address this first in \cref{sec:catlyap} for flows and then in the sequel \cite{CLT2} for systems with dynamics, but we begin by motivating this observation. Informally, the simplest system that is stable is $\dot{y} = 0$ for $y \in \Rplus$, and a Lyapunov function can be expressed by a positive definite function $V \maps E \to \Rplus$ such that the following diagram commutes up to inequality: 
\begin{eqnarray}
\label{eqn:introclassiclyap}
\begin{tikzcd}
    E \times \R^n 
    \ar[r, "\T V"]
    &
    \Rplus \times \R 
    \\
    E   
    \ar[r, "V"'] 
    \ar[u, "\vec{f}\coloneq(\id_{E}{,}f)"]
    &
    \Rplus
    \ar[ul, phantom, "\leq"]
    \ar[u, "\vec{0} = (\id_{\Rplus}{,}0)"']
\end{tikzcd}
\end{eqnarray}
where $\T V(x,y) = \left( V(x),\frac{\partial V}{\partial x}  \vert_{x} y \right)$. 
The diagram commuting up to inequality yields: 
\begin{eqnarray}
    \T V \circ \vec{f} \leq \vec{0} \circ V & \iff &  \T V(x,f(x)) \leq (V(x),0) \nonumber\\
    \label{eqn:Lyapupconditions}
    & \iff &  \frac{\partial V}{\partial x} \Big \vert_{x} f(x) \leq 0 \\
    & \iff  & \dot{V}(x) \leq 0 \nonumber
\end{eqnarray}
thus recovering \eqref{eqn:lyapunov_condition}. 
Invoking the comparison lemma \cite[Lemma 3.4]{NonlinearSystems} we conclude that $V(c(t)) \leq V(c(0))$, which can be expressed in diagram form as follows.
\begin{eqnarray}
\label{eqn:stabilitydiagrammotivation}
\begin{tikzcd}
    \Rplus
    \arrow[r, "c"]
    \arrow[d, "c(0)"']
    &
    E
    \arrow[d, "V"]
    \\
    E
    \arrow[r, "V"']
    &
    \Rplus
    \ar[ul, phantom, "\geq"]
\end{tikzcd}
\end{eqnarray}
We thus obtain a description of Lypaunov functions in terms of categorical diagrams. 

\subsection{Overview: Categorical Approach to Stability}

In order to provide a categorical framework for Lyapunov stability, we begin by distilling Lyapunov theory to its most basic elements. That is, before considering Lyapunov functions on systems with dynamics, we instead begin with formulating Lyapunov functions as they act on collections of solutions, i.e., flows. The result is a set of axioms for a \define{setting for stability} in an ambient category $\C$, with key components: 
\begin{itemize}
    \item a \emph{time object} $T$
    \item an \emph{object of interest} $E$
    \item a \emph{measurement object} $R$
\end{itemize}
Here, $T$ is a time object (equipped with the structure of a monoid), $E$ is the space of interest represented as an object in $\C$, and $R$ is a simple stable object where measurements can be compared. The axioms assert a partial order on the hom-set $\C(E,R)$ which allows for morphisms to be compared via morphisms between morphisms, $\To$, and hence enables diagrams that lax commute, i.e., ``commute up to inequality.'' Through this higher categorical structure ($\C$ can equivalently be viewed as a 2-category, with $R$ a posetal object), distances can be formulated and this leads to a notion of ``norm'' $\| \cdot \|_{x^*}$ that is positive and point separating. Therefore, a setting for stability has the necessary ingredients to ``do Lyapunov theory.''

We begin by generalizing stability---as represented diagrammatically in \eqref{eqn:stabilityintro}---to settings for stability. This is done at the level of flows: $\flow \maps T \times E \to E$, with $\flow$ an action of the monoid $T$ on $E$. An equilibrium point $x^* \maps 1 \to E$, with $1$ the terminal object of $\C$, is simply a fixed point of the flow as represented by a commuting diagram. Importantly, this formulation does not require $E$ to have the underlying structure of a set, i.e., the ambient category $\C$ need not be concrete. 
A stable equilibrium point $x^* \maps 1 \to E$ can then be defined as the following diagram lax commuting: 
\[
\begin{tikzcd}
    T \times E 
    \arrow[rr, "\flow"]
    \arrow[d, "\pi"']
    &&
    E
    \arrow[d, "\|\cdot\|_{x^*}"]
    \\
    E
    \arrow[r, "\|\cdot\|_{x^*}"']
    \arrow[urr, Rightarrow, ""]
    &
    R
    \arrow[r, "\alpha"']
    &
    R
    \end{tikzcd}
\]
with $\alpha$ a \emph{class $\K$ morphism}---order preserving and distinguished point preserving, with inverse. This, therefore, generalizes \eqref{eqn:stabilityalphaintro} to a categorical context. 

The key generalization of Lyapunov theory is that of a \define{Lyapunov morphism}: $V \maps E \to R$ operating in a setting for stability. As with classic Lyapunov functions \eqref{eqn:lyapunov_condition}, this must be positive definite and decresent. These conditions can be expressed via lax commuting diagrams: 
\[
\begin{array}{c}
\textit{Positive Definite:} \\
\begin{tikzcd}
    &
    R
    \arrow[dr, "\overline \alpha"]
    \\
    E
    \arrow[ur, "\|\cdot\|_{x^*}"]
    \arrow[rr, "V"{name = V}, ""'{name = B}]
    \arrow[dr, "\|\cdot\|_{x^*}"']
    &&
    R
    \\&
    R
    \arrow[ur, "\underline \alpha"']
    \arrow[from = 1-2, to = V, Rightarrow, ""]
    \arrow[from = B, to = 3-2, Rightarrow, ""]
\end{tikzcd}
\end{array}
\qquad \qquad 
\begin{array}{c}
\textit{Decresent:} \\
\begin{tikzcd}
    T \times E
    \arrow[r, "\flow"]
    \arrow[d, "\pi"']
    &
    E
    \arrow[d, "V"]
    \\
    E
    \arrow[r, "V"']
    \arrow[ur, Rightarrow]
    &
    R
\end{tikzcd}
\end{array}
\]
for $\overline{\alpha}$ and $\underline{\alpha}$ class $\K$ morphisms. The positive definite condition implies the positivity of $V$, $V \To 0$, along with point separating: $\ker V \cong 1 \xrightarrow{x^*} E$. The decresent condition generalizes the diagram \eqref{eqn:stabilitydiagrammotivation} representing Lyapunov functions in the classic setting. We establish that, in a setting for stability, the existence of a Lyapunov morphism implies stability of flows. Conversely, we prove that stable flows imply the existence of a Lyapunov morphism---thereby showing Lyapunov morphisms are necessary and sufficient for stability, and that the axioms underlying a setting for stability characterize stability. This is illustrated on a range of examples, from classic stability for continuous and discrete time systems, to enriched categories. 

\subsection{Next Steps: A View Toward Systems}

There is something missing from the Lyapunov approach to stability based purely on flows. In the classical setting, the flow of a system is obtained by solving the differential equation from all initial conditions. If this were possible, there would be no need to use Lyapunov functions as certificates of stability---one could simply directly check the flow for stability. Indeed, one could say that the entire point of Lyapunov's theorem is to give a characterization of stable equilibrium points in terms of the pointwise behavior of the system encoded by its dynamics. That is, Lyapunov's theorem connects the local instantaneous behavior of a system---the rate of change of the Lyapunov function, cf. \eqref{eqn:lyapunov_condition}---to the stability behavior of its flows---staying within a bounded neighborhood of the equilibrium point, cf. \eqref{eqn:stabilitydef}. 

We resolve this disconnect in \cite{CLT2}, where we relate the global flow perspective to the dynamics of systems. Specifically, given a functor $\F \maps \C \to \C$, we view systems as $\F$-coalgebras, i.e., morphisms $\vec f \maps E \to \F(E)$. Given a setting for stability, we present additional axioms that yield a \emph{setting for dynamic stability} and, in this setting, obtain the following Lyapunov characterization of stability: 

\begin{thm*}[Generalized Lyapunov's theorem, \cite{CLT2}]
    Let $\vec f \maps E \to \F(E)$ be an $\F$-coalgebra in a setting for dynamic stability. An equilibrium point $x^* \maps 1 \to E$ of $\vec f$ is stable if and only if there exists a morphism $V \maps E \to R$ such that 
    \begin{enumerate}
        \item $V$ is positive definite with respect to $x^*$, and
        \item the following diagram lax commutes.
        \[\begin{tikzcd}
            \F(E)
            \arrow[r, "\F(V)"]
            &
            \F(R)
            \\
            E
            \arrow[u, "\vec f"] 
            \arrow[r, "V"']
            &
            R
            \arrow[ul, Rightarrow]
            \arrow[u, "\vec{0}_R"']
        \end{tikzcd}\]
    \end{enumerate}
\end{thm*}
This generalizes the diagram in \eqref{eqn:introclassiclyap} to systems, framed as $\F$-coalgebras, where here $\vec{0}_R \maps R \to \F(R)$ is the ``zero'' system on the stable object. 

\section{Settings for Stability}
\label{sec:catlyap}

The classical version of Lyapunov's theorem states that the property of an equilibrium point being stable is equivalent to the existence of a positive definite and decrescent function---\emph{a Lyapunov function}. The main result of this paper is a categorical generalization of Lyapunov's theorem. The sort of generalization we mean is not to a larger class of dynamical systems. Instead, we are generalizing to ``dynamical systems'' within different categories, where classical dynamical systems are such within the category of Euclidean subspaces, or more broadly, manifolds.

In order to state this generalized Lyapunov theorem, we must lay out precisely what an appropriate setting within a category looks like, one which enables us to discuss systems, their solutions, and comparison of solutions. In the present work, we focus on studying a system through its total solution flow. Roughly speaking, this means our state space should carry an action of a monoid representing the flow of time, e.g., $\Zplus$ or $\Rplus$ for discrete or continuous time systems respectively. We also make use of a space of values which serves as the codomain for measurements, such as norms or Lyapunov functions. We ask this space to be equipped with a partial order with which we compare measurements. These ingredients form what we call a \emph{setting for stability}, the axiomatic framework for which our generalized Lyapunov theorem holds. 

\subsection{Settings for Stability}
\label{sec:setting}

In this section, we lay out the necessary definitions to make rigorous the notion of a setting for stability alluded to in the introduction. The corresponding axioms give a framework for ``doing Lyapunov theory.'' Before presenting these, we review some necessary preliminaries.

\begin{defn}
\label{def:monoid}
    Let $\C$ be a category with finite products. A \define{monoid} $(T, \oplus, 0_\oplus)$ internal to $\C$ is an object $T$ equipped with a multiplication map $\oplus \maps T \times T \to T$ and a unit map $0_\oplus \maps 1 \to T$ such that the following diagrams commute: 
    \[
    \begin{tikzcd}
        T \times T \times T
        \arrow[r, "\id_T \times \oplus"]
        \arrow[d, "\oplus \times \id_T"']
        & 
        T \times T 
        \arrow[d, "\oplus"] 
        \\
        T \times T
        \arrow[r, "\oplus"']
        & 
        T
    \end{tikzcd}
    \qquad\qquad
    \begin{tikzcd}
        T
        \arrow[r, "\id \times 0_\oplus"]
        \arrow[dr, "\id_T"']
        &
        T \times T
        \arrow[d, "\oplus"]
        &
        T
        \arrow[l, "0_\oplus \times id"']
        \arrow[dl, "\id_T"]
        \\&
        T
    \end{tikzcd}
    \]
\end{defn}

\begin{rmk}
    Monoid objects in $\Set$ recover the traditional notion of monoid. Monoid objects are generally definable in a monoidal category. The additional properties afforded by a cartesian monoidal structure are not strictly necessary.
\end{rmk}

\begin{defn}
\label{def:kernel}
    Let $\C$ be a category with a terminal object $1$, $X, Y$ objects of $\C$, and $Y$ equipped with a distinguished point $0_Y \maps 1 \to Y$. The \define{kernel} of a map $f \maps X \to Y$ in $\C$ is the equalizer of $f$ and the composite $0: X \xrightarrow !1 \xrightarrow{0_Y} Y$. 
    \[
    \begin{tikzcd}[column sep = small]
        \ker(f) 
        \arrow[r, "eq_f"]
        &
        X
        \arrow[rr, "f", shift left=1]
        \arrow[rr, "0"', shift right=1]
        \arrow[dr, "!"']
        &&
        Y
        \\&&
        1
        \arrow[ur, "0_Y"']
    \end{tikzcd}
     \]
\end{defn}
Note that even though we refer to the object $\ker(f)$ as the kernel, the data of the equalizing map $eq_f$ is absolutely necessary.

\begin{defn}
\label{def:posetal}
    A \define{posetal object} $R$ in a category $\C$ is one such that for each object $X \in \C$, the hom-set $\C(X, R)$ is equipped with a partial order such that for any map $f \maps X \to Y$, if $g_1 \geq g_2$ in $\C(Y, R)$, then $g_1 \circ f \geq g_2 \circ f$ in $\C(X, R)$.
\end{defn}

This definition is motivated by the case when the objects of $\C$ are something like sets equipped with structure, and $R$ itself is equipped with a partial order. In this scenario, a partial order is induced on any hom-set $\C(X, R)$ pointwise.

The definition of posetal object implies that $\C$ is equipped with the structure of a category enriched in the category of posets and order-preserving maps, with the discrete order assigned to any hom-set $\C(X, Y)$ with $Y \neq R$. There is nothing in the above definition that would impose coherence in the presence of multiple posetal objects. In light of this observation, we are justified in expressing inequalities $g_1 \geq g_2$ in these hom-posets as unlabeled 2-dimensional morphisms $g_1 \To g_2$, i.e., lax commuting diagrams. The diagram below demonstrates the order-preservation of left-whiskering in \cref{def:posetal}.
\[\begin{tikzcd}
    X
    \arrow[r, "f"]
    &
    Y
    \arrow[r, bend left, "g_1"]
    \arrow[r, phantom, "\Downarrow"]
    \arrow[r, bend right, "g_2"']
    &
    R
\end{tikzcd}
\qquad 
\mapsto
\qquad 
\begin{tikzcd}[column sep = small, row sep = small]
    &
    Y
    \arrow[dd, Rightarrow]
    \arrow[dr, "g_1"]
    \\
    X
    \arrow[ur, "f"]
    \arrow[dr, "f"']
    &&
    R
    \\&
    Y
    \arrow[ur, "g_2"']
\end{tikzcd}
\]
Beginning with the proof of \cref{lem:normproperties}, we frequently construct 2-morphisms as above via ``pasting diagrams''. A rigorous and self-contained exposition of the calculus of pasting diagrams can be found in section A.4 of Schommer-Pries' thesis \cite{SchommerPriesThesis}.

\begin{rmk}
    A more careful definition of posetal object asks $\C$ to be a 2-category, and the appropriate hom-categories to have the property (as opposed to structure as above) of being a poset. In the present work, we do not consider more than one posetal object at a time, and thus delay such a careful treatment for future work. 
\end{rmk}

\begin{rmk}
\label{rmk:locallyordered}
    If $R$ is a poset considered as a category, then a natural transformation of the form $\alpha \maps F \To G \maps X \to R$ has components which are uniquely determined by their domain and codomain. Thus such a natural transformation is unique if it exists, making the category $\Cat(X, R)$ of functors and natural transformations a poset. This is the motivating example for the definition of posetal object. 
\end{rmk}

We are now ready to present the main definition of this paper: an axiomatic framework for Lyapunov theory.

\begin{defn}
\label{def:setting}
    A \define{setting for stability} in $\C$, denoted $(\C,T,E,R,\To,d)$ or $(T,E,R)$ when this does not lead to confusion, consists of:
    \begin{enumerate}
    \setcounter{enumi}{-1}
        \item[\textbf{S0:}] (setting) a category $\C$ with all finite products
        \item[\textbf{S1:}] (space) an object $E \in \C$
        \item[\textbf{S2:}] (time) an object $T$ equipped with the structure of a monoid: $(T, \oplus, 0_\oplus)$ (see \cref{def:monoid})
        \item[\textbf{S3:}] (stable object) a posetal object $R \in \C$ (see \cref{def:posetal}) equipped with a distinguished point $0_R \maps 1 \to R$ 
        \item[\textbf{S4:}] (distance) a morphism in $\C$, denoted by $
        d \maps E \times E \to R$, such that:
        \begin{itemize}
            \item $d \To 0$, 
            \item $\ker(d) \cong \Delta \maps E \to E \times E$ (see \cref{def:kernel}), with $\Delta$ the diagonal map.
        \end{itemize} 
        This can be depicted in a single lax commuting diagram. 
        \begin{eqnarray}
        \label{eqn:distanceaxiom}
        \begin{tikzcd}
            E
            \arrow[r, "\Delta"]
            &
            E \times E
            \arrow[r, bend left, "d", ""'{name = A}, pos = 0.4]
            \arrow[r, phantom, "\Downarrow", pos = 0.35]
            \arrow[r, bend right, "0"', ""{name = B}, pos = 0.45]
            &
            R
            \\
            \ker(d)
            \arrow[u, "\cong"]
            \arrow[ur, "eq"']
        \end{tikzcd}
        \end{eqnarray}
    \end{enumerate}
\end{defn}

\begin{rmk}
    To illustrate the constructions throughout this section, consider the case where $\C$ is a concrete category and therefore objects have the underlying structure of a set, morphisms the underlying structure of a function. As a result, for a setting for stability in $\C$ we can consider points $a, b \in T$, $x, y \in E$ and the distinguished point $0_R \in R$. We will illustrate the constructions on this example when presenting new concepts. For example, the conditions on the distance morphism in \textbf{S4} can be stated in this setting as:  
    \begin{itemize}
        \item $d \To 0$, \ex{$d(x, y) \geq 0$}
        \item $\ker(d) \cong \Delta \maps E \to E \times E$. \ex{$d(x, y) = 0$ iff $x = y$}
    \end{itemize} 
\end{rmk}

\begin{defn}
\label{def:norm}
    Given a point $x^* \maps 1 \to E$, the \define{norm} $\|\cdot\|_{x^*} \maps E \to R$ relative to $x^*$ is the following composite.
    \[
    \|\cdot\|_{x^*} \maps E \xrightarrow{id_E \times x^*} E \times E \xrightarrow d R
    \exmath{\| x \|_{x^*} \coloneq d(x, x^*)}
    \]
\end{defn}

\begin{lem}
\label{lem:normproperties}
    A norm $\|\cdot\|_{x^*} \maps E \to R$ satisfies: 
    \begin{itemize}
        \item Positivity: $ \|\cdot\|_{x^*} \To 0$, 
        \ex{$\| x \|_{x^*} \geq 0$}
        \item Point-separating: \[
        \begin{tikzcd}
            \ker(\|\cdot\|_{x^*})
            \arrow[r, "eq"]
            \arrow[d, "\cong"']
            &
            E
            \\
            1
            \arrow[ur, "x^*"']
        \end{tikzcd}
        \exmath{\| x \|_{x^*} = 0 \quad \textrm{iff} \quad   x = x^*}
        \]
    \end{itemize}  
\end{lem}
\begin{proof}
    Positivity follows from the following lax commuting diagram.
    \[
    \begin{tikzcd}
        E
        \arrow[rrr, bend left = 50, "\|\cdot\|_{x^*}"]
        \arrow[r, "id_E \times x^*"]
        \arrow[drr, "!"', bend right = 30]
        &
        E\times E
        \arrow[rr, "d", ""'{name = A}]
        \arrow[dr, "!"']
        &&
        R
        \\&&
        1
        \arrow[ur, "0_R"', bend right = 10]
        \arrow[from = A, to = 2-3, Rightarrow]
    \end{tikzcd}
    \]
    Point-separating:
    We can see that $(1, x^*)$ does equalize $\|\cdot\|_{x^*}$ and $0$ by the following diagram:
    \[
    \begin{tikzcd}
        &
        E
        \arrow[drrr, "\|\cdot\|_{x^*}"]
        \arrow[dr, "\id_E \times x^*"description]
        \\
        1
        \arrow[ur, "x^*"]
        \arrow[dr, "x^*"']
        &&
        E \times E
        \arrow[rr, "d"description]
        &&
        R
        \\&
        E
        \arrow[ur, "\Delta"]
        \arrow[r, "\Delta"']
        &
        E \times E
        \arrow[urr, "0"description]
        \arrow[r, "!"']
        &
        1
        \arrow[ur, "0_R"']
    \end{tikzcd}
    \]
    The lower path is clearly $0 \circ x^*$.
\end{proof}

\subsection{Positive Definite Morphisms}

\begin{defn}
    A map $\alpha \maps R \to R$ is \define{order-preserving} if right-whiskering, i.e., $\alpha_* \maps \C(X, R) \to \C(X, R)$ given by $\alpha_*(f) = \alpha \circ f$, is order-preserving on the hom-posets. In other words, if $f \geq g \maps X \to R$, then $\alpha \circ f \geq \alpha \circ g$.
\end{defn}
Diagrammatically, the condition above can be stated as:
\[\begin{tikzcd}
    X
    \arrow[r, bend left, "f"]
    \arrow[r, phantom, "\Downarrow"]
    \arrow[r, bend right, "g"']
    &
    R
    \arrow[r, "\alpha"]
    &
    R
\end{tikzcd}
\qquad 
\mapsto
\qquad 
\begin{tikzcd}[column sep = small, row sep = small]
    &
    R
    \arrow[dd, Rightarrow]
    \arrow[dr, "\alpha"]
    \\
    X
    \arrow[ur, "f"]
    \arrow[dr, "g"']
    &&
    R
    \\&
    R
    \arrow[ur, "\alpha"']
\end{tikzcd}\]

\begin{defn}
\label{def:classKmorphism}
    A morphism $\alpha \maps R \to R$ is of \define{class $\K$} if:
    \begin{itemize}
        \item $\alpha$ is an order-preserving map 
        \item $\alpha$ has an order-preserving inverse $\alpha\inv$
        \item $\alpha \circ 0_R = 0_R$.
    \end{itemize} 
\end{defn}

\begin{rmk}
    We do not ask for our class $\K$ morphisms to be continuous. In our formulation, this is a property that should be supplied by the category of interest. Continuous class $\K$ functions are the class $\K$ morphisms in $\Top$, the category of topological spaces and continuous maps.
\end{rmk}

\begin{lem}
    If $\alpha \maps R \to R$ is class $\K$, then so is its inverse $\alpha\inv \maps R \to R$. 
\end{lem}
\begin{proof}
    Trivially follows since $\alpha$ and $\alpha^{-1}$ are both order preserving by definition and: $\alpha \circ 0_R = 0_R$ iff $\alpha^{-1} \circ \alpha \circ 0_R = \alpha^{-1} \circ 0_R$ iff $0_R = \alpha^{-1} \circ 0_R$. 
\end{proof}

\begin{defn}
\label{def:positivedef}
    A morphism $V \maps E \to R$ is \define{positive definite} with respect to $x^* \maps 1 \to E$ if there exist class $\K$ morphisms $\underline\alpha, \overline \alpha \maps R \to R$ such that the following diagram lax commutes.
    \[\begin{tikzcd}
        &
        R
        \arrow[dr, "\overline \alpha"]
        \\
        E
        \arrow[ur, "\|\cdot\|_{x^*}"]
        \arrow[rr, "V"{name = V}, ""'{name = B}]
        \arrow[dr, "\|\cdot\|_{x^*}"']
        &&
        R
        \\&
        R
        \arrow[ur, "\underline \alpha"']
        \arrow[from = 1-2, to = V, Rightarrow, ""]
        \arrow[from = B, to = 3-2, Rightarrow, ""]
    \end{tikzcd}
    \exmath{\underline{\alpha}(\| x \|_{x^*}) \leq V(x) \leq \overline{\alpha}(\| x\|_{x^*})}
    \]
\end{defn}

\begin{lem}
  If $V \maps E \to R$ is positive definite with respect to a point $x^* \maps 1 \to E$ then: 
  \begin{itemize}
    \item $V \To 0$   \ex{$V(x) \geq 0$}
    \item $\ker V  \cong 1 \xrightarrow{x^* = eq} E$ 
    \ex{$V(x) = 0$ iff $x = x^*$}
\end{itemize}
\end{lem}

\begin{proof}
    For the first assertion, consider the following diagram, which is a pasting of the definition and \cref{lem:normproperties} \[\begin{tikzcd}
        E
        \arrow[rr, "V"{name = V}, ""'{name = B}]
        \arrow[dr, "\|\cdot\|_{x^*}"description, ""'{name = C}]
        \arrow[dr, "0"'{name = D}, bend right=50, ""]
        &&
        R
        \\
        &
        R
        \arrow[ur, "\underline \alpha"']
        \arrow[from = C, to = D, Rightarrow, ""]
        \arrow[from = B, to = 2-2, Rightarrow, ""]
    \end{tikzcd}
    \]
    where the bottom path is clearly equal to 0. For the second assertion, assume $V$ is bounded by class $\K$ functions. We can see that $V(x^*)=0$ by the following pasting diagram, where the frames without a 2-cell filler commute strictly.
    
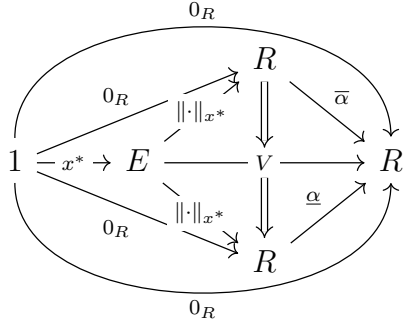
\[
    \begin{tikzcd}
        &&
        R
        \arrow[dr, "\overline \alpha"]
        \\
        1
        \arrow[urr, "0_R"]
        \arrow[r, "x^*"description]
        \arrow[drr, "0_R"']
        \arrow[rrr, bend left = 90, "0_R"]
        \arrow[rrr, bend right = 90, "0_R"']
        &
        E
        \arrow[ur, "\|\cdot\|_{x^*}"description]
        \arrow[rr, "V"description, ""{name = A}, ""'{name = B}]
        \arrow[dr, "\|\cdot\|_{x^*}"description]
        &&
        R
        \\&&
        R
        \arrow[ur, "\underline \alpha"]
        \arrow[from = 1-3, to = A, Rightarrow]
        \arrow[from = B, to = 3-3, Rightarrow]
    \end{tikzcd}
    \]
    If $V(x) = 0$, then $\|x\|_{x^*}=0$ by the following pasting diagram:
    
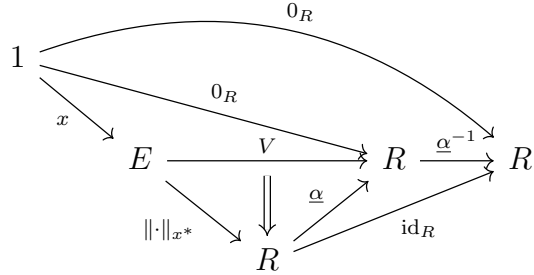
\[
    \begin{tikzcd}
        1
        \arrow[drrrr, "0_R", bend left]
        \arrow[drrr, "0_R"]
        \arrow[dr, "x"']
        \\&
        E
        \arrow[rr, "V", ""'{name = A}]
        \arrow[dr, "\|\cdot\|_{x^*}"']
        &&
        R
        \arrow[r, "\underline \alpha\inv"]
        &
        R
        \\&&
        R
        \arrow[ur, "\underline\alpha"]
        \arrow[from = A, to = 3-3, Rightarrow]
        \arrow[urr, "\id_R"']
    \end{tikzcd}
    \]
   and thus $x=x^*$. Hence $V$ is positive definite. 
\end{proof}

\subsection{Flows}

In this section, we choose to take as a starting point systems which admit a global flow, so that they may be represented by an action of a monoid on the space. This leads to the notions of equilibrium point and stable point relative to this data defined in the next section.

\begin{defn}
\label{def:flow}
    A \define{$T$-flow} $\flow \maps T \times E \to E$ on an object $E$ is an action of $T$ on $E$, meaning the following diagrams commute: 
    \[
    \begin{array}{c}
    \textit{Initialization:} \\
    \begin{tikzcd}
        E
        \arrow[r, "0_\oplus \times \id_E"]
        \arrow[dr, "\id_E"']
        &
        T \times E 
        \arrow[d, "\flow"]
        \\&
        E
    \end{tikzcd}
    \\
    \excenter{\flow(0, x) = x}
    \end{array}
    \qquad\qquad
    \begin{array}{c}
    \textit{Composition:} \\
    \begin{tikzcd}
        T \times T \times E
        \arrow[r, "\oplus \times \id_E"]
        \arrow[d, "\id_T \times \flow"']
        &
        T \times E
        \arrow[d, "\flow"]
        \\
        T \times E 
        \arrow[r, "\flow"']
        &
        X
    \end{tikzcd} 
    \\
    \excenter{\flow(t_1, \flow(t_2, x)) = \flow(t_1 + t_2, x)}
    \end{array}
    \]
    A \define{morphism of $T$-flows} $p \maps \flow_X \to \flow_Y$ is a map $p \maps X \to Y$ such that the following diagram commutes.
    \[
    \begin{tikzcd}
        T \times X
        \arrow[r, "\id_T \times p"]
        \arrow[d, "\flow_X"']
        &
        T \times Y
        \arrow[d, "\flow_Y"]
        \\
        X 
        \arrow[r, "p"']
        &
        Y
    \end{tikzcd}
    \exmath{p (\flow_X(t,x) )= \flow_Y(t,p(x))}
    \]
    Let $\TFlow$ denote the category of $T$-flows and morphisms of $T$-flows.
\end{defn}

\begin{rmk}
Note that the definition above is of a \emph{left} action of $T$ on $E$. There is also a notion of \emph{right} actions of monoids, but right actions can be converted into left actions and vice versa, such that they are equivalent. This is due to our choice of working in cartesian categories, in particular symmetric monoidal categories. 
\end{rmk}

\section{Categorical Lyapunov Theorem}

This section will establish that, given a setting for stability, Lyapunov morphisms give a complete characterization of the stability. To this end, we first formulate stability of equilibria categorically, followed by Lyapunov morphisms $V \maps E \to R$: morphisms that are positive definite and decrescent. We then prove that the existence of a Lyapunov morphism is both sufficient and necessary for the stability of equilibria.

\subsection{Equilibria and Stability}

For all subsequent definitions, assume a setting for stability, and $E$ be equipped with a flow $\flow \maps T \times E \to E$.

\begin{defn}
\label{def:equlibriumpoint}
    An element $x^* \maps 1 \to E$ is an \define{equilibrium point} if the following diagram commutes:
    \[
    \begin{tikzcd}
        T
        \arrow[r, "id_T \times x^*"]
        \arrow[d, "!"']
        &
        T \times E
        \arrow[d, "\flow"]
        \\
        1
        \arrow[r, "x^*"']
        &
        E
    \end{tikzcd}
    \exmath{\flow_t(x^*) = x^*}
    \]
\end{defn}

\begin{rmk}
    The commutativity of the diagram given in \cref{def:equlibriumpoint} is equivalent to that of the following diagram.
    \[\begin{tikzcd}
        T
        \arrow[r, "\id_T \times x^*"]
        \arrow[d, "\id_T \times x^*"']
        &
        T\times E
        \arrow[d, "\flow"]
        \\
        T \times E
        \arrow[r, "\pi_E"']
        &
        E
    \end{tikzcd}\]
\end{rmk}

\begin{defn}
\label{def:stable}
    An equilibrium point $x^* \maps 1 \to E$ is \define{stable} if there is a class $\K$ morphism $\alpha$ such that the following diagram lax commutes:  
    \[
    \begin{tikzcd}
        T \times E 
        \arrow[rr, "\flow"]
        \arrow[d, "\pi"']
        &&
        E
        \arrow[d, "\|\cdot\|_{x^*}"]
        \\
        E
        \arrow[r, "\|\cdot\|_{x^*}"']
        \arrow[urr, Rightarrow, ""]
        &
        R
        \arrow[r, "\alpha"']
        &
        R
    \end{tikzcd}
    \exmath{\| \flow(t, x) \|_{x^*} \leq \alpha(\|x\|_{x^*})}
    \]
\end{defn}

\begin{rmk}
    Note that Lyapunov theory is generally concerned with the forward trajectory of a point in a system, and thus the appropriate time objects tend to be positive monoids, e.g., $\Rplus$ or $\Zplus$. The framework presented here permits groups such as $\R$ or $\Z$ to take the role of time object, but the resulting notions of stability may be unfamiliar. In practice, it is often possible to find a suitable sub-monoid of a group time object. For instance, if $T$ is an ordered group, taking the \define{positive cone} $T_{\geq0} = \{t \in T \mid 0 \leq t\}$ may be a good time object, as previously mentioned. This definition can be stated internally via a taking a slice over $0$.
\end{rmk}

\begin{defn}
\label{def:weaklycontracting}
    A flow $\flow \maps T \times E \to E$ is said to be \define{weakly contracting} if the following diagram lax commutes, where $\sigma$ denotes the canonical symmetry map. 
    \[
    \begin{array}{c}
    \begin{tikzcd}
        T \times E \times E 
        \arrow[r, "\pi"]
        \arrow[d, "\Delta \times \id"']
        &
        E \times E
        \arrow[r, "d"]
        \arrow[d, Rightarrow]
        &
        R
        \\
        T \times T \times E \times E
        \arrow[r, "\id \times \sigma \times \id"']
        & 
        T \times E \times T \times E
        \arrow[r, "\flow \times \flow"']
        &
        E\times E 
        \arrow[u, "d"]
    \end{tikzcd}
    \\
    \excenter{d(\flow_t(x),\flow_t(y)) \leq d(x,y)}
    \end{array}
    \]
\end{defn}

\begin{prop}
\label{prop:weaklycontracting}
    If $\flow$ is weakly contracting, then any equilibrium point is stable.
\end{prop}
\begin{proof}
Let $x^* \maps 1 \to E$ be an equilibrium point. The diagram below demonstrates that $x^*$ must be stable, with class $\K$ function $\alpha = \id_R$.
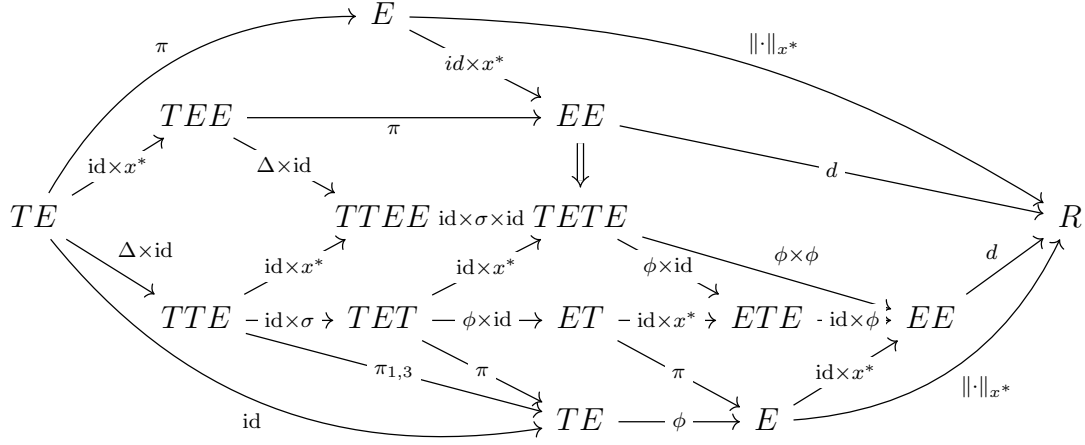
\[\begin{tikzcd}
	&& E \\
	& TEE && EE \\
	TE && TTEE & TETE &&& R \\
	& TTE & TET & ET & ETE & EE \\
	&&& TE & E
	\arrow["{id \times x^*}"{description}, from=1-3, to=2-4]
	\arrow["\norm", bend left = 15, from=1-3, to=3-7]
	\arrow["\pi"', from=2-2, to=2-4]
	\arrow["{\Delta \times \id }"{description}, from=2-2, to=3-3]
	\arrow[shorten <=2pt, shorten >=2pt, Rightarrow, from=2-4, to=3-4]
	\arrow["d"{description}, from=2-4, to=3-7]
	\arrow["\pi", bend left, from=3-1, to=1-3]
	\arrow["{\id \times x^*}"{description}, from=3-1, to=2-2]
	\arrow["{\Delta \times \id }", from=3-1, to=4-2]
	\arrow["\id"', bend right, from=3-1, to=5-4]
	\arrow["{\id \times \sigma \times \id }"{description}, from=3-3, to=3-4]
	\arrow["{\flow \times \id }"{description}, from=3-4, to=4-5]
	\arrow["{\flow \times \flow}", from=3-4, to=4-6]
	\arrow["{\id \times x^*}"{description}, from=4-2, to=3-3]
	\arrow["{\id \times  \sigma}"{description}, from=4-2, to=4-3]
	\arrow["{\pi_{1,3}}"{description}, from=4-2, to=5-4]
	\arrow["{\id \times x^*}"{description}, from=4-3, to=3-4]
	\arrow["\flow\times \id "{description}, from=4-3, to=4-4]
	\arrow["\pi"{description}, from=4-3, to=5-4]
	\arrow["{\id \times x^*}"{description}, from=4-4, to=4-5]
	\arrow["\pi"{description}, from=4-4, to=5-5]
	\arrow["{\id\times \flow}"{description}, from=4-5, to=4-6]
	\arrow["d", from=4-6, to=3-7]
	\arrow["\flow"{description}, from=5-4, to=5-5]
	\arrow["\norm"', bend right, from=5-5, to=3-7]
	\arrow["{\id \times x^*}"{description}, from=5-5, to=4-6]
\end{tikzcd}\]
The only lax commuting part of the diagram comes from the condition that $\flow$ is weakly contracting. All the strictly commuting parts are straightforward manipulations of the cartesian monoidal structure, except one which utilizes the fact that $x^*$ is an equilibrium point.
\end{proof}

\subsection{Categorical Lyapunov Theorem}

A setting for stability gives the ability to reframe Lyapunov's classic theorem (on flows) categorically via Lyapunov morphisms. 

\begin{defn}
\label{def:lyapunov}
    A morphism $V \maps E \to R$ is called a \define{Lyapunov morphism} for $\flow$ with equilibrium point $x^* \maps 1 \to E$ if:
    \begin{enumerate}
        \item (positive definite) $V$ is bounded by class $\K$ morphisms
        \[\begin{tikzcd}
            &
            R
            \arrow[dr, "\overline \alpha"]
            \\
            E
            \arrow[ur, "\|\cdot\|_{x^*}"]
            \arrow[rr, "V"{name = V}, ""'{name = B}]
            \arrow[dr, "\|\cdot\|_{x^*}"']
            &&
            R
            \\&
            R
            \arrow[ur, "\underline \alpha"']
            \arrow[from = 1-2, to = V, Rightarrow, ""]
            \arrow[from = B, to = 3-2, Rightarrow, ""]
        \end{tikzcd}
        \exmath{V(x) \geq 0, V(x) = 0 \text{ iff } x = x^*.}
        \]
        
        \item (decrescent) the following diagram lax commutes.
        \[
        \begin{tikzcd}
            T \times E
            \arrow[r, "\flow"]
            \arrow[d, "\pi"']
            &
            E
            \arrow[d, "V"]
            \\
            E
            \arrow[r, "V"']
            \arrow[ur, Rightarrow]
            &
            R
        \end{tikzcd}
        \exmath{ V( \flow(t, x) ) \leq V(x) }
        \]
    \end{enumerate}
\end{defn}

\begin{thm}[Categorical Lyapunov theorem]
\label{thm:Lyap}
    In a setting for stability let $\flow \maps T \times E \to E$ be a flow, and $x^* \maps 1 \to E$ an equilibrium point. If there exists a Lyapunov morphism $V$ for $\flow$, then $x^*$ is a stable equilibrium point. 
\end{thm}
\begin{proof}
    Since we assume $V$ to be positive definite, by Definition \ref{def:positivedef}, there exist class $\K$ functions $\overline \alpha, \underline \alpha \maps R \to R$ such that the following diagram lax commutes.
    \[\begin{tikzcd}
        &
        R
        \arrow[dr, "\overline \alpha"]
        \\
        E
        \arrow[ur, "\|\cdot\|_{x^*}"]
        \arrow[rr, "V"{name = V}, ""'{name = B}]
        \arrow[dr, "\|\cdot\|_{x^*}"']
        &&
        R
        \\&
        R
        \arrow[ur, "\underline \alpha"']
        \arrow[from = 1-2, to = V, Rightarrow, ""]
        \arrow[from = B, to = 3-2, Rightarrow, ""]
    \end{tikzcd}\]
    Therefore, we can construct the following lax commuting diagram: 
    \[
    \begin{tikzcd}
        &
        E
        \arrow[r, "\|\cdot\|_{x^*}"]
        \arrow[drr, bend right = 15, ""{name = A}, "V"description]
        \arrow[dd, Rightarrow]
        &
        R
        \arrow[dr, "\overline\alpha"]
        \\
        T \times E
        \arrow[ur, "\pi_E"]
        \arrow[dr, "\flow"']
        &&&
        R
        \arrow[dr, "\underline\alpha\inv"]
        \\&
        E
        \arrow[r, "\|\cdot\|_{x^*}"']
        \arrow[urr, bend left = 15, ""'{name = B}, "V"description]
        &
        R
        \arrow[ur, "\underline \alpha"']
        \arrow[rr, "id"', ""{name = C}]
        &&
        R
        \arrow[from = 1-3, to = A, Rightarrow]
        \arrow[from = B, to = 3-3, Rightarrow]
    \end{tikzcd}
    \]
    It follows that $x^*$ is stable, with the corresponding class $\K$ function being $\underline \alpha\inv \circ \overline \alpha$.
\end{proof}

\subsection{Converse Lyapunov Theorem}

The generalized Lyapunov theorem of the previous section holds under very weak assumptions. The converse requires some assumptions on the posetal object $R$: the existence and property of suprema of certain maps valued in $R$. The construction of a Lyapunov morphism for a stable equilibrium point is essentially given by defining $V(x)$ to be the largest distance from the stable point achieved by $\flow(t,x)$. Here we encode this idea categorically in a point-free manner. 

For the definitions below, assume a setting for stability in $\C$. 

\begin{defn}
\label{def:upperboundAoff}
    For $A,B \in \C$, an \define{upper-bound for $A$ of} $f \maps A \times B \to R$ is a morphism $\alpha \maps B\to R$ such that the following diagram lax commutes.
    \[\begin{tikzcd}
        &
        B
        \ar[dr, "\alpha"]
        \\
        A\times B
        \ar[rr, ""{name=U, above}, "f"']
        \ar[ur, "\pi_B"]
        &&
        R
        \ar[to=U, Rightarrow, from=1-2]
    \end{tikzcd}
    \exmath{f(a,b) \leq \alpha(b)}
    \]
\end{defn}

\begin{defn}
\label{def:localsuprema}
    We say that $R$ has \define{local suprema} if whenever $f \maps A\times B\to R$ has an upper-bound, then there is a lowest upper-bound $\sup_Af \maps B \to R$. That is, for any upper-bound $\beta \maps B\to R$, we have
    \[
    \begin{array}{c}
    \begin{tikzcd}
        &
        B
        \ar[dr, "\beta"]
        \\
        A\times B
        \ar[rr, ""{name=U, above}, "f"']
        \ar[ur, "\pi_B"]
        &&
        R
        \ar[to=U, Rightarrow, from=1-2]
    \end{tikzcd} \\
    \excenter{f(a,b) \leq \beta(b)}
    \end{array}
    \quad \text{implies}\quad
        \begin{array}{c}
    \begin{tikzcd}
        &
        B
        \ar[dr, "\sup_Af"', sloped, ""{name = B}]
        \ar[dr, "\beta", bend left = 70, ""'{name = A}]
        \\
        A\times B
        \ar[rr, ""{name=U, above}, "f"']
        \ar[ur, "\pi_B"]
        &&
        R
        \ar[to=U, Rightarrow, from=1-2]
        \arrow[from = A, to = B, Rightarrow]
    \end{tikzcd} \\
    \excenter{f(a,b) \leq \sup_Af(b) \leq  \beta(b)}
    \end{array}
    \]
\end{defn}

\begin{defn}
    We say the \textbf{supremum commutes with whiskering} if: 
    \[
    (\sup_A f) \circ r = \sup_A (f \circ (\id_A \times r))
    \]
    for all $r \maps C \to B$. That is, in the following diagram the composite of the top line is a supremum of the composite of the bottom line.
    \[
    \begin{array}{c} 
    \begin{tikzcd}
    	C 
        & 
        B 
        \\
    	{A \times C} 
        & 
        {A\times B} 
        & 
        R
    	\arrow[""{name=0, anchor=center, inner sep=0}, "r", from=1-1, to=1-2]
    	\arrow[""{name=1, anchor=center, inner sep=0}, "\sup_Af", bend left, from=1-2, to=2-3]
    	\arrow[from=2-1, to=1-1, "\pi_C"]
    	\arrow[""{name=2, anchor=center, inner sep=0}, "id \times r"', from=2-1, to=2-2]
    	\arrow["{\pi_B}", from=2-2, to=1-2]
    	\arrow[""{name=3, anchor=center, inner sep=0}, "f"', from=2-2, to=2-3]
    	\arrow[shorten <=4pt, shorten >=4pt, Rightarrow, from=0, to=2]
    	\arrow[shorten <=3pt, shorten >=3pt, Rightarrow, from=1, to=3]
    \end{tikzcd}\\
    \excenter{\sup_A f(r(c)) = \sup_A f(a,r(c))}
    \end{array}
    \]
\end{defn}

The converse conditions presented here were developed in collaboration with S\'ebastien Mattenet building upon \cite{MattenetJungers}. This collaboration led to the follow-on paper \cite{CLT2}.

\begin{thm}[Converse categorical Lyapunov theorem]
\label{thm:Lyapconv}
    Assume a setting for stability such that $R$ has local suprema commuting with whiskering. Let $\flow \maps T \times E \to E$ be a $T$-flow, and $x^* \maps 1 \to E$ an equilibrium point. If $x^*$ is a stable equilibrium point then there exists a Lyapunov morphism $V$ for $\flow$.
\end{thm}
\begin{proof}
Given a stable equilibrium point $x^* \maps 1 \to E$, the diagram lax commuting: 
\[\begin{tikzcd}
     T\times E 
     \ar[rr, "\flow"]
     \ar[d, "\pi_E"', ""{name=0, anchor=center, inner sep=3}] 
     && 
     E
     \ar[d, "\norm", ""{name=1, anchor=center, inner sep=0}]
     \\
     E
     \ar[urr, Rightarrow]
     \ar[r, "\norm"']
     &
     R
     \ar[r, "\alpha"']
     &
     R
\end{tikzcd}\] 
tells us that $\xnorm{\flow}$ is upper-bounded by $\alpha(\norm)$, and so we get a map:
\[
V \coloneq \sup_T\xnorm{\flow} \qquad \mathrm{with} \qquad 
\begin{tikzcd}
    &
    E
    \ar[dr, "V"description, ""{name = B}]
    \ar[dr, "\alpha(\norm)", bend left = 70, ""'{name = A}]
    \\
    T \times E
    \ar[rr, ""{name=U, above}, "\xnorm \flow"']
    \ar[ur, "\pi_E"]
    &&
    R
    \ar[to=U, Rightarrow, from=1-2]
    \arrow[from = A, to = B, Rightarrow]
\end{tikzcd}.
\]
We conclude with a few whiskering operations. First we whisker by $0_\oplus\in T$:
\[\begin{tikzcd}        
    &&
    E
    \ar[ddr, "\alpha(\norm)"]
    \ar[d, Rightarrow]
    \\&&
    E
    \ar[dr, "V"description]
    \\
    E
    \ar[r, "0_\oplus\times\id_E"]
    &
    T\times E
    \ar[rr, "\xnorm{\flow}"', ""{name=U, above}]
    \ar[ur, "\pi_E"']
    \ar[uur, "\pi_E"]
    &&
    R
    \ar[to=U, Rightarrow, from=2-3]
\end{tikzcd}\] 
This gives us that $V$ is positive definite, since 
\[
\pi_E \circ (0_\oplus \times \id_E) = \id_E
\qquad \mathrm{and} \qquad
\flow \circ (0_\oplus \times \id_E) = \id_E. 
\]
That is, the result is the following lax commuting diagram. 
\[\begin{tikzcd}
     &
     R
     \ar[dr, "\alpha"]
     \\
     E
     \ar[ur, "\norm"]
     \ar[rr, "V"{description}, ""{name=0, anchor=center, inner sep=0}]
     \ar[dr, "\norm"']
     &&
     R
     \\&
     R
     \ar[ur, "\beta= \id_R"']
     \ar[Rightarrow, from=1-2, to=0, shorten =3pt]
     \ar[Rightarrow, from=0, to=3-2, shorten =3pt]
\end{tikzcd}\]

It remains to show that $V$ decreases along the flow. Since the supremum commutes with whiskering, $V \circ \flow$ is the supremum of $\norm \circ \flow \circ (\id_T \times \flow)$ over $T$, or equivalently the supremum of $\norm \circ \flow\circ (\oplus \times \id)$ over $T$. That is, the following lax commutes.
\[\begin{tikzcd}
    T \times E 
    \ar[r, "\flow"]
    & 
    E 
    \ar[drr, "V", ""{name=v, anchor=center}, bend left = 20]
    \\
    T \times T \times E
    \ar[u, "\pi_{2, 3}"]
    \ar[r, "\id\times \flow"']
    \arrow[dr, "\oplus \times \id_E"']
    & 
    {T \times E} 
    \ar[u, "\pi_{E}"]
    \ar[r, "\flow"]
    &
    E
    \ar[r, "\norm"']
    & 
    R
    \arrow[from = v, to = 2-3, Rightarrow]
    \\&
    T \times E
    \arrow[ur, "\flow"']
\end{tikzcd}\]

The inequality $V(\pi_E)\geq \xnorm{\flow}$ can be whiskered by $\oplus \times \id_E$:
\[\begin{tikzcd}
    &
    T \times E
    \ar[r, "\pi_E"]
    &
    E
    \ar[dr, "V"]
    \ar[d, Rightarrow]
    \\
    T\times T\times E
    \ar[ur, "\pi_{2,3}"]
    \ar[r, "\oplus \times \id_E"']
    &
    T\times E
    \ar[ur, "\pi_E"]
    \ar[r, "\flow"']
    &
    E
    \ar[r, "\norm"']
    &
    R
\end{tikzcd}\]
This shows that $V\circ \pi_E$ is an upper bound of $\norm \circ \flow \circ (\oplus \times \id)$. By minimality of $\sup_T\norm\circ \flow\circ (\oplus\times \id_E)=V\circ \flow$, we have
    \[\begin{tikzcd}
        T\times E 
        \ar[r, "\flow"]
        \ar[d, "\pi_E"', ""{name=0, anchor=center, inner sep=3}]
        &
        E
        \ar[d, "V", ""{name=1, anchor=center, inner sep=0}]
        \\
        E
        \ar[ur, Rightarrow]
        \ar[r, "V"']
        &
        \Rplus 
    \end{tikzcd}\]
    giving the desired inequality for $V$ to be decrescent.
\end{proof}

\section{Examples of Settings for Stability}
\label{sec:examples}

We illustrate constructions of settings for stability via a few key examples. 

\subsection{Classical Stability}
\label{sec:classicalstability}

Consider an ODE $\dot{x} = f(x)$ with $x \in E$ and $E$ and open subset of $\mathbb{R}^n$. Under suitable assumptions, the solution to the ODE is given by a flow:
\begin{eqnarray}
    \label{eqn:flow}
    \flow \maps \mathbb{R}_{\geq 0} \times E  \to E
\end{eqnarray}
satisfying $\dot{\flow}_t(x) = f(\flow_t(x))$, where $x \in E$ is the initial condition. 
This leads to a setting for stability as considered in the introduction, where
\begin{enumerate}
\setcounter{enumi}{-1}
    \item[\textbf{S0:}] (setting) the category $\Set$, with terminal object $1 = \{*\}$. 
    \item[\textbf{S1:}] (space) the space of interest is a set $E \subset \mathbb{R}^n \in \Set$. 
    \item[\textbf{S2:}] (time) the object $\mathbb{R}_{\geq 0}$ is a monoid with $+$ addition and $0_\oplus = 0$. 
    \item[\textbf{S3:}] (stable object) the set $\mathbb{R}_{\geq 0}$ of the non-negative reals, with  
    \begin{itemize}
        \item distinguished point: $0 \in \Rplus$. 
        \item posetal structure: $f \To g$, if $f(x) \geq g(x)$ for all $x \in E$. 
    \end{itemize} 
     \item[\textbf{S4:}] (distance) the morphism:
    \begin{eqnarray}
        d \maps E \times E  & \to &  \Rplus \\
        (x,y) & \mapsto &  \| x - y \| \nonumber
    \end{eqnarray}
    where $\|\cdot\|$ is any (classical) norm in $\mathbb{R}^n$. As a result, the diagram \eqref{eqn:distanceaxiom} lax commutes since: 
    \begin{itemize}
        \item $d(x,y) \geq 0$
        \item  $d(x,y) = 0$ iff $x = y$
    \end{itemize}  
\end{enumerate}
    
Per this setting for stability, we can ``do Lyapunov theory'' per the constructions in \cref{sec:setting}. 

\paragraph{\textbf{Norm.}}  The distance yields yields the (categorical) norm per Definition \ref{def:norm}: 
\[
\|\cdot\|_{x^*} \coloneqq \| (~\cdot~) - x^*\| \maps E \to R
\]
which satisfies: 
\begin{itemize}
    \item $ \|\cdot\|_{x^*} \To 0$ since $\| x \|_{x^*} \geq 0$,
    \item  $\ker(\|\cdot\|_{x^*}) \cong \{x^*\}$ since $\| x \| = 0$ iff  $x = x^*$
\end{itemize}  

\paragraph{\textbf{Class $\K$ Functions.}}  Per Definition \ref{def:classKmorphism}, a class $K$ morphism $\alpha: \Rplus \to \Rplus$ satisfies: 
\begin{itemize}
    \item $\alpha$ is strictly increasing on $\Rplus$: $r_1 < r_2$ implies that $\alpha(r_1) < \alpha(r_2)$, 
    \item $\alpha$ maps $0 \in \Rplus$ to $0$. 
\end{itemize}
therefore, class $\K$ morphisms in this setting are class $\K_{\infty}$ functions as traditionally defined. 

\paragraph{\textbf{Positive Definite Functions.}}  Per Definition \ref{def:positivedef}, $V \maps E \to \Rplus$ is positive definite if there exists class $\K$ maps $\underline{\alpha},\overline{\alpha}$ such that:  
\[
\underline{\alpha}(\| x \|_{x^*} ) \leq V(x) \leq 
\overline{\alpha}( \| x \|_{x^*} )
\]
It is well-known that this is equivalent to $V$ being positive definite in the usual sense. 

\paragraph{\textbf{Flows.}} Flows per Definition \ref{def:flow} are classical flows \eqref{eqn:flow}. The fact that $\flow \maps \mathbb{R}_{\geq 0} \times E  \to E$ is an action of $\Rplus$ on $E$ implies that the following is satisfied: 
\begin{eqnarray}
    \textit{Initialization:} &  & \flow_{0}(x) = x  \nonumber\\
    \textit{Composition:} &  &   
    \nonumber \flow_{t_1} (\flow_{t_2}(x)) = \flow_{t_1 + t_2}(x) 
\end{eqnarray}

\paragraph{\textbf{Stability.}}  An equilibrium point per Definition \ref{def:equlibriumpoint} simply satisfies: $\flow_t(x^*) = x^*$. This equilibrium point is stable, per Definition \ref{def:stable} if: 
\[
\| \flow_t(x) \|_{x^*} \leq \alpha(\| x \|_{x^*})
\]
for a class $\K$ map $\alpha$. Note that this agrees with the ``classical'' notion of stability:

\begin{lem}
\label{lem:classicalstability}
    If $x^*$ is stable in the sense of \cref{def:stable}, then $x^*$ is stable in the classical sense\footnote{These are actually equivalent, but we only provide the proof for the direction stated above.}, i.e., for all $\epsilon > 0$, there exists a $\delta > 0$ such that: 
    \[
    \| x - x^* \| < \delta \qquad \mathrm{implies} \qquad 
    \| \flow_t(x) - x^* \| < \epsilon  \quad \forall t \in \Rplus
    \]
\end{lem}
\begin{proof}
    For a given $\epsilon > 0$, let $\delta = \alpha^{-1}(\epsilon)$. 
    If $\| x - x^* \|  = \| x\|_{x^*} \leq \delta$, then we have 
    \begin{align*}
       \| \flow_t(x) - x^* \| =   \| \flow_t(x) \|_{x^*} \leq \alpha(\| x \|_{x^*}) \leq \alpha(\delta) = \alpha(\alpha^{-1}(\epsilon)) = \epsilon.
    \end{align*}
\end{proof}

\paragraph{\textbf{Lyapunov's Theorem.}} 
A Lyapunov morphism in this setting is a positive definite function such that $V(\flow_t(x)) \leq V(x)$. \cref{thm:Lyap} along with \cref{lem:classicalstability} implies that this condition implies stability, i.e., we recover \eqref{eqn:stabilitydiagrammotivation}. 
Moreover, it will be seen in Proposition \ref{prop:lyapunov} that when $\flow_t$ is the solution of an ODE we recover the classic Lypaunov conditions given in \eqref{eqn:lyapunov_condition}.

\paragraph{\textbf{Converse Lyapunov Theorem.}}
A function $f \maps A \times B \to \Rplus$ having an upper bound on $A$ in the sense of \cref{def:upperboundAoff} means that for each fixed $b \in B$, there is a constant $\alpha(b) \in \Rplus$ which bounds the function $f(-,b) \maps A \to \Rplus$ from above. A supremum of such a function would be a function $\sup_Af \maps B \to \Rplus$ with the property that $f(a,b) \leq \sup_Af(b)$ for all $b \in B$, and if $\beta \maps B \to \Rplus$ is an upper-bound for $A$ of $f$, then $\sup_Af(b) \leq \beta(b)$ for all $b \in B$. This obviously coincides with the traditional definition of the supremum. The condition of asking the function $f$ to have an upper-bound for $A$ of $f$ guarantees that as a function, $\sup_Af$ achieves only finite values, legitimizing the codomain of $\Rplus$. Thus, $\Rplus$ has local suprema. That these suprema commute with whiskering is trivial.

Let $\flow \maps T \times E \to E$ be a flow and $x^*$ be a stable equilibrium point. The proof of \cref{thm:Lyapconv} gives a specific construction of a Lyapunov function for $x^*$. The function $\xnorm\flow$ gets an upper-bound for $T$ directly from the stability of $x^*$. By \cref{def:stable}, there exists a class $K$ function $\alpha \maps \Rplus \to \Rplus$ such that $\xnorm{\flow_t(x_0)} \leq \alpha(\xnorm{x_0})$. So $\alpha \circ \norm$ is the upper-bound for $A$ of $\xnorm\flow$. Therefore, for all $x_0 \in E$, the function $\xnorm{\flow_-(x_0)} \maps T \to \Rplus$ is bounded above by the constant $\alpha(\xnorm{x_0})$. We can thus define the Lyapunov function $V$ to be the supremum \[V \coloneq \sup_{T} \xnorm\flow \maps E \to \Rplus.\]
which gives the pointwise formula:
\[V(x) \coloneq \sup_{t \in T} \xnorm{\flow_t(x)}.
\]

\subsection{Discrete-Time Stability}
\label{sec:discrete-time}

Consider a discrete-time dynamical system $x_{k+1} = F(x_k)$ with $x_k \in E \subset \mathbb{R}^n$ and $F \maps E \to E$. This results in a flow: 
\begin{eqnarray}
\label{eqn:discretetime}
    \flow \maps \mathbb{Z}_{\geq 0} \times E & \to & E \\
    (k,x) & \mapsto & F^k(x) \nonumber
\end{eqnarray}
Here, $F^0(x) = x$ (the initial condition), and $F^k(x) = F(F(\ldots (F(x)\ldots)))$ application of $F$ $k$ times, i.e., $F^0 = id$, $F^1 = F$, $F^2 = F \circ F$, $F^k = F \circ F^{k-1}$. This leads to a setting for stability defined by: 
\begin{enumerate}
\setcounter{enumi}{-1}
    \item[\textbf{S0:}] (setting) the category $\Set$, with terminal object $1 = \{*\}$. 
    \item[\textbf{S1:}] (space)  the space of interest is the set $E \subset \mathbb{R}^n \in \Set$.
    \item[\textbf{S2:}] (time) the object $\mathbb{Z}_{\geq 0}$ is a monoid with $+$ addition and $0_\oplus = 0$. 
    \item[\textbf{S3:}] (stable object) the set $\mathbb{R}_{\geq 0}$ of the non-negative reals as with the continuous-time case.
    \item[\textbf{S4:}] (distance) the morphism $d \maps E \times E \to R$ as in the continuous time case, induced by the Euclidean metric. 
\end{enumerate}

The constructions in this case are the same as for the continuous-time case, except the \textbf{flow} is now given as in \eqref{eqn:discretetime}. This implies that an \textbf{equilibrium point} $x^*$ is a fixed point: $F(x^*) = x^*$. As a result, $x^*$ is \textbf{stable} if: 
\[
\| \flow_k(x) \|_{x^*} = \| F^k(x) \|_{x^*}  \leq \alpha(\| x \|_{x^*}). 
\]
From this, we recover the classical notion of discrete-time stability, with the proof the same as Lemma \ref{lem:classicalstability}. 

\begin{lem}
\label{lem:discretetimestability}
    If $x^*$ is stable, then for all $\epsilon > 0$, there exists a $\delta > 0$ such that: 
        \[
    \| x - x^* \| < \delta \qquad \mathrm{implies} \qquad 
    \| F^k(x) - x^* \| < \epsilon  \quad \forall k \in \mathbb{Z}_{\geq 0}
    \]
\end{lem}

\paragraph{\textbf{Lyapunov's Theorem.}}  A Lyapunov morphism in this case satisfies: 
\begin{enumerate}
    \item $\underline{\alpha}(\| x \|_{x^*} ) \leq V(x) \leq \overline{\alpha}( \| x \|_{x^*} )$
    \item $V(\flow_k(x)) \leq V(x)$. 
\end{enumerate}
Theorem \ref{thm:Lyap} implies that, in that case, $x^*$ is a stable equilibrium point. The second condition (for $k = 1$) gives the classic Lyapunov condition on a discrete-time dynamical system: 
\[
\nabla V(x) \coloneq V(F(x)) - V(x) \leq 0. 
\]
Therefore, we recover the classic Lyapunov conditions for discrete-time dynamical systems from Theorem \ref{thm:Lyap}

\paragraph{\textbf{Converse Lyapunov Theorem.}} The stable object and ambient category are the same as in \cref{sec:classicalstability}, therefore the converse also holds.

\begin{prop}
\label{prop:discretetimelyap}
    For the discrete-time dynamical system $x_{k+1} = F(x_k)$ with $x_k \in E \subset \mathbb{R}^n$, and equilibrium point $x^*$, i.e., a fixed point $F(x^*) = x^*$ is stable if and only if there exists a function $V \maps E \to \Rplus$ satisfying: 
    \begin{enumerate}[(i)]
    \item $V$ is positive definite: $V(x) \geq 0$ and $V(x) = 0$ iff $x = x^*$. 
    \item $\nabla V(x)$ is negative semi-definite: $\nabla V(x) = V(F(x)) - V(x) \leq 0$. 
\end{enumerate}
\end{prop}

\subsection{Smooth Manifolds}
\label{sec:smoothmanifolds}

Consider the category $\Man$ of smooth manifolds with boundary and smooth maps\footnote{The construction of the Lyapunov function given in our proof of the converse does not produce a smooth map, but this map can be smoothed to yield the desired converse result.}. As addition is smooth, the space $\Rplus$ is a monoid object in $\Man$, i.e., a \define{Lie monoid} \cite{Grad23}. Its actions $\flow \maps \Rplus \times M \to M$ are precisely smooth global flows on the manifold $M$. 

\begin{enumerate}
\setcounter{enumi}{-1}
    \item[\textbf{S0:}] (setting) $\C = \Man$ is the category of smooth manifolds with boundary and smooth maps.
    \item[\textbf{S1:}] (space) a smooth manifold $M$ with empty boundary.
    \item[\textbf{S2:}] (time) $\Rplus$ with monoid structure given by addition $+$ and unit $0$.
    \item[\textbf{S3:}] (stable object) $\R$ with:
            \begin{itemize}
            \item point $0 \in \R$. 
            \item postal structure: $f \To g$, if $f(x) \geq g(x)$ for all $x \in M$. 
            \end{itemize}
    \item[\textbf{S4:}] (distance) It is well-known that every smooth manifold admits a Riemannian metric $g_p \maps T_p M \times T_p M \to \R$, this Riemannian metric can be used to define an ordinary metric by taking the infimum of lengths of curves connecting points \cite{abraham2012manifolds}:
    \begin{eqnarray}
        d \maps M \times M  & \to &  \Rplus \\
        (x,y) & \mapsto &  \inf\{ L(\gamma) \maps \gamma \textrm{~admissible with}~ \gamma(0) = x ,~ \gamma(1) = y\}. 
        \nonumber
    \end{eqnarray}
    where an admissible curve $\gamma \maps [0,1] \to M$ is a differentiable curve such that $\dot{\gamma}(t) \in T_{\gamma(t)} M $, and $L(\gamma)$ is the length: 
    \[
    L(\gamma) = \int_{0}^{1} \sqrt{g_{\gamma(t)}(\dot{\gamma}(t),\dot{\gamma}(t))} dt. 
    \]
    The topology induced by this metric agrees with the original topology of the manifold. 
\end{enumerate}

In this setting, a flow is precisely a smooth global flow on $M$, equilibrium points are what one would expect, and
class $\K$ morphisms in this setting are class $\K$ functions which are additionally smooth. Lyapunov morphisms can be characterized for vector fields as the following results indicates. 

\begin{prop}
\label{prop:lyapunov}
    Let $M$ be a manifold and $\flow\colon\Rplus\times M\to M$ a smooth action of $(\Rplus,0,+)$, and let $f\maps M\to TM$ be the vector field given by
    \[f(x)\coloneqq \frac{\partial\flow}{\partial t}\Big|_{(0,x)}=\lim_{t\to 0}\frac{\flow(t,x)}{t}.\] 
    For a positive definite function $V\maps M\to\Rplus$, the vector field formula $\vec{0}\geq TV\circ f$ holds iff we have $V(x)\geq V(\flow(t,x))$ for all $t,x$. That is, $V$ is a Lypaunov morphism (the square to the left lax commutes) iff $V$ is a Lyapunov function (the square to the right laxly commutes):
    \[
    \begin{tikzcd}
        \Rplus\times M\ar[r, "\flow"]\ar[d, "\pi"']&
        M\ar[r, "f"]\ar[d, "V"']&
        TM\ar[d, "TV"]\\
        M\ar[r, "V"']
        \arrow[ur, Rightarrow]
        &
        \Rplus\ar[r, "\vec{0}"']
        \arrow[ur, Rightarrow]
        &
        T\Rplus
    \end{tikzcd}
    \]
\end{prop}
\begin{proof}
For all $x,t$, we have $V(x)\geq V(\flow(t,x))$ iff $0\geq V(\flow(t,x))-V(x)$. So if $V(x)\geq V(\flow(t,x))$ then, noting that $x=\flow(0,x)$, dividing both sides by $t>0$ and taking the limit, we have
\[
0\geq\lim_{t\to 0}\frac{(V\circ\flow)(t,x))-(V\circ\flow)(0,x)}{t}=TV\circ f.
\]
In the other direction, one uses the comparison lemma \cite{NonlinearSystems} applied to $\dot{y} = 0$ for $y \in \Rplus$ with initial condition $y(0) = V(x)$ yielding: 
\[
\dot{V}(x) = TV \circ f \leq 0 \qquad \textrm{implies} \qquad V(\flow(t,x)) \leq V(x)
\]
as desired. 
\end{proof}

Note that this result recovers the conditions of the classic Lyapunov theorem \eqref{eqn:lyapunov_condition}; namely, we have recovered the fact that the diagram in \eqref{eqn:introclassiclyap} lax commuting implies \eqref{eqn:stabilitydiagrammotivation} commutes, and vice versa. We generalize this idea to systems in \cite{CLT2}. 

\subsection{Kalman filtering}
\label{sec:Kalman}

Settings for stability can involve unconventional time, space and stability objects. An example of this is given by Kalman Filtering, a method widely used to estimate the state of uncertain systems. Consider a linear (discrete-time) system subject to state and output noise: 
\begin{eqnarray}
    x_{k+1} & = & A_k x_k + F_k w_k \nonumber\\
    y_k & = & C_k x_k + v_k \nonumber
\end{eqnarray}
where $A_k \in \R^{n \times n}, F_k \in \R^{n \times p}, C_k \in \R^{q \times n}$, with $A_k$ non-singular, are deterministic time-varying matrices describing the process and output dynamics with process noise $w_k$ and measurement noise $v_k$ represented by random vectors that are Gaussian with 0 mean and (for simplicity) identity covariance, $w_k, v_k \in \mathcal{N}(0,I)$. 

Let $\hat{x}_k$ be the best (optimal) estimate of $x_k$, i.e., the expected value $\mathbb{E}[x_k - \hat{x}_k] = 0$, and $P_k$ be the conditional error covariance matrix, $P_k \coloneq \mathbb{E}[(x_k - \hat{x}_k)(x_k - \hat{x}_k)^T]$. The central result by K\'alm\'an \cite{kalman1960new} is that $\hat x_k$ can be obtained as a linear combination of $\hat x_{k-1}$ and $y_k$:
\begin{eqnarray}
    \hat{x}_k & = &  (A_k - P_k R_k A_k) \hat{x}_{k-1} + P_k C_k^T y_k \nonumber\\
    P_k & = & \flow_k(P_{k-1}) \nonumber
\end{eqnarray}
where $R_k = C_k^T C_k$ and $\phi_k$ is given by the expression
\begin{eqnarray}
\label{eqn:errorcovariance}
    \flow_k(P) \coloneq 
    (A_k P A_k^T + S_k) (I + R_k S_k + R_k A_k P A_k^T)^{-1}
\end{eqnarray}
where $S_k = F_k F_k^T$. We can study the convergence properties of the Kalman filter from a Lyapunov perspective. 

Following \cite{Bougerol93}, denote by $J$ the symplectic matrix:
\[
J=
\begin{bmatrix}
0 & I\\-I & 0
\end{bmatrix},
\]
where $I$ is the $n\times n$ identity matrix and $0$ is the $n\times n$ zero matrix. The symplectic group $Sp(n)$ is the group of all $2n\times 2n$ matrices $M$ satisfying $M^T JM=J$. We also represent such matrices $M$ by:
\[
M=\begin{bmatrix} A & B\\
C & D
\end{bmatrix},
\]
where $A,B,C,D\in \R^{n\times n}$. We now define the set of matrices:
\[
\mathcal{H} \coloneq \{ M \in Sp(n) ~ : ~ A ~ \textrm{invertible}, \quad BA^T \in \mathcal{P}, 
\quad A^TC  \in \mathcal{P} \} \]
where $\mathcal{P}$ is the set of positive semi-definite $n \times n$ matrices (for all $P \in \mathcal{P}$, if $x \neq 0$ then $x^T P x \geq 0$). 
It can be shown that $\mathcal{H}$ is a monoid under matrix multiplication. 

Let $\mathcal{P}_0$ be the cone of positive definite $n \times n$ matrices (for all $P \in \mathcal{P}_0$, if $x \neq 0$ then $x^T P x > 0$). We now define an action of $\mathcal{H}$ on $\mathcal{P}_0$ by:
\begin{eqnarray}
\label{eqn:kalmanflow}
    \flow \maps \mathcal{H}  \times \mathcal{P}_0  & \to &  \mathcal{P}_0  \\
    (M,P) & \mapsto & (AP+B)(CP+D)^{-1} \nonumber
\end{eqnarray}
This flow is termed the \emph{the discrete Riccati equation}, and encodes solutions to the classic continuous-time matrix Riccati equation. Namely, under mild assumptions, if $P_t$, $t \in \Rplus$, is the solution to the Riccati equation for a continuous-time system: 
\[
\dot{P}_t = A_t P_t + P_t A_t^T - P_t R_t P_t + S_t, \qquad P_0 \in \mathcal{P}_0
\]
for $R_t,S_t \in \mathcal{P}_0$, $t \in \Rplus$, there exists a family of $N_t \in \mathcal{H}$ such that $P_t = \flow(N_t,P_0)$. 

This leads to a setting for stability defined by: 
\begin{enumerate}
\setcounter{enumi}{-1}
    \item[\textbf{S0:}] (setting) the category $\Man$, with terminal object $1 = \{*\}$. 
    \item[\textbf{S1:}] (space) The space of interest is the cone of positive-definite matrices $\mathcal{P}_0$. 
    \item[\textbf{S2:}] (time) the object $\mathcal{H}$, which is a monoid with matrix multiplication and $0_\oplus = I$ the identity matrix. 
    \item[\textbf{S3:}] (stable object) the set $\mathbb{R}_{\geq 0}$ of the non-negative reals, with  
    \begin{itemize}
        \item point $0 \in \Rplus$. 
        \item posetal structure: $f \To g$, if $f(x) \geq g(x)$ for all $x \in E$. 
    \end{itemize} 
    \item[\textbf{S4:}] (distance) the morphism:
    \begin{eqnarray}
    \label{eqn:kalmandistance}
        d \maps \mathcal{P}_0  \times \mathcal{P}_0  & \to &  \Rplus  \\
        (P,Q) & \mapsto & d(P,Q)=\sqrt{\sum_{i=1}^n \log^2\left(\lambda_i(PQ^{-1}) \right)} \nonumber
    \end{eqnarray}
    where $\lambda_1,\hdots,\lambda_n$ are the eigenvalues of $PQ^{-1}$.
    The diagram \eqref{eqn:distanceaxiom} lax commutes since: 
    \begin{itemize}
        \item $d(P,Q) \geq 0$
        \item  $d(P,Q) = 0$ iff $P = Q$
    \end{itemize}  
\end{enumerate}

\paragraph{\textbf{Flows.}}  A flow $\flow \maps \mathcal{H} \times \mathcal{P}_0 \to \mathcal{P}_0$ is given as in \eqref{eqn:kalmanflow}. 
In \cite{Bougerol93} (see Theorem 1.7) it was shown that $\flow$ is a \emph{contraction}:
for all $M \in \mathcal{H}$ and $P,Q \in \mathcal{P}_0$:
\begin{eqnarray}
\label{eqn:contraction}
    d(\flow(M,P),\flow(M,Q))\leq d(P,Q).
\end{eqnarray}
The goal is to leverage this result in the categorical framework for stability.

\paragraph{\textbf{Stability.}}  An \emph{equilibrium point} $P^* \in \mathcal{P}_0$ is a fixed point: \[\flow(M,P^*) = (AP^* + B)(CP^* + D)\inv = P^*\] for any $M \in \mathcal H$. As a result, the positive definite matrix $P^*$ is \emph{stable} if: 
\[
 \| \flow(M,P) \|_{P^*}  \leq \alpha(\| P \|_{P^*}). 
\]
where the \emph{norm} is given by: 
$\|  P \|_{P^*} = d(P,P^*)$ as given by \eqref{eqn:kalmandistance}. Here $\alpha$ is a smooth class $\K$ function since the stable object is $\Rplus$. 

\paragraph{\textbf{Lyapunov's Theorem.}}  A Lyapunov morphism $V \maps \mathcal{P}_0 \to \Rplus$ for $\flow$ in this case satisfies: 
\begin{enumerate}
    \item positive definite: $\underline{\alpha} (\|P\|_{P^*}) \leq V(P) \leq \overline{\alpha}(\|P\|_{P^*})$
    \item decrescent: $V(\flow(M,P)) \leq V(P)$. 
\end{enumerate}
Theorem \ref{thm:Lyap} implies that, in that case, $P^*$ is a stable equilibrium point. The following proposition implies that Lyapunov morphisms exist in this setting for stability and are, in fact, trivial. 

\begin{prop}
\label{prop:kalman}
    Consider the flow $\flow \maps \mathcal{H} \times \mathcal{P}_0 \to \mathcal{P}_0$ given in \eqref{eqn:kalmanflow} with equilibrium point $P^* \in \mathcal{P}_0$. Then the morphism $V \maps \mathcal{P}_0 \to \Rplus$ defined by
    \[
    V(P) \coloneq \| P \|_{P^*} = d(P,P^*)
    \]
    is a Lyapunov morphism for $\flow$, and therefore the positive definite matrix $P^*$ is stable. 
\end{prop}
\begin{proof}
    By the definition of $V$, picking $\underline{\alpha} = \overline{\alpha} = \id_{\Rplus}$ yields (1). By \eqref{eqn:contraction}:
    \[
    V(\flow(M,P)) = d(\flow(M,P),P^*) = 
    d(\flow(M,P),\flow(M,P^*)) \leq d(P,P^*) = V(P)
    \]
    Therefore, (2) is satisfied.
\end{proof}

To see the importance of stability in the context of Kalman filtering, note that the flow $\flow$ in \eqref{eqn:kalmanflow} corresponds to the error covariance matrix updates as determined by \eqref{eqn:errorcovariance}. Namely,
\begin{eqnarray}
\label{eqn:Mk}
\flow(M_k, P) = \flow_k(P) \qquad \mathrm{with} \qquad 
M_k  = \begin{bmatrix}
    A_k & S_k A^{-T}_k \\
    R_k A_k & (I + R_k S_k) A^{-T}_k  
\end{bmatrix} \in \mathcal{H}. 
\end{eqnarray}
Therefore, stability of the equilibrium point $P^*$ of the flow $\flow$, as certified by the Lyapunov morphisms $V$, implies that the error covariance matrix $P_k =  \mathbb{E}[(x_k - \hat{x}_k)(x_k - \hat{x}_k)^T]$ is bounded. As a result, by Proposition \ref{prop:kalman}, the estimation error $x_k - \hat{x}_k$ is bounded (in expectation) when utilizing a Kalman filter. 

\section{Stability in Enriched Categories}
\label{sec:EnrichedStability}

In his seminal work \cite{LawvereMetric}, Lawvere pointed out that metric spaces are special cases of categories enriched in the non-negative real numbers, what is now referred to as \emph{Lawvere metric spaces}. Following this example, enriched categories give a natural framework for the synthesis of settings for stability. We begin by briefly reviewing the basics of enriched category theory in \cref{sec:enrichedcats}. The reader will find further detail in Kelly's book \cite{Kelly}. We then give the setting for stability in a Lawvere metric space in \cref{sec:LawvereMetric} as motivation for the general case of stability in a category enriched in a closed symmetric monoidal poset with products in \cref{sec:StabEnrichedCat}. This then allows us to specialize to the case of set-based stability by enriching in a power set in \cref{sec:powersetenriched}.

\subsection{Enriched Categories}
\label{sec:enrichedcats}

We begin with a review of some basic terminology regarding enriched categories. 

\begin{defn}
\label{def:enrichedcat}
    Let $(\V, \otimes, 0_{\V})$ be a monoidal category with monoidal product $\otimes$ and monoidal unit $0_{\V}$. A \define{category enriched in $\V$} or \define{$\V$-category}, $\C$, consists of the following data:
    \begin{itemize}
        \item a set $ob(\C)$ of \define{objects}
        \item for all $X, Y \in ob(\C)$, a \define{hom object} $\C(X,Y) \in \V$ 
        \item for all $X, Y, Z \in C$ a \define{composition morphism} in $\V$: 
        \[
            \C(X,Y) \otimes \C(Y,Z) \xrightarrow{\circ} \C(X,Z)
        \]
        \item for all $X \in \C$, an \define{identity morphism} in $\V$: 
        \[
        0_{\V} \xrightarrow{\iota_X} \C(X,X)
        \]
    \end{itemize}
    The composition morphism and identity morphism are subject to consistency conditions \cite{Kelly}.
\end{defn}

 If $\V$ is additionally closed and symmetric, then $\V$ itself may be given the structure of a $\V$-category via internal hom.

\begin{rmk}
    A $(\Set, \times, 1)$-enriched category is precisely a (small) category. Historically, many early examples of enriched categories had $\V$ a concrete category, i.e., a category of sets equipped with extra structure, e.g., abelian groups or vector spaces. In these cases, a $\V$-category is then precisely an ordinary category equipped with extra structure on the hom sets, hence ``enriched''. However, there are monoidal categories which are not concrete, but we may consider categories enriched in them anyway, as in \cref{sec:LawvereMetric} below. In this case, it does not make sense to think of a $\V$-category as an ordinary category equipped with extra structure.
\end{rmk}

\begin{defn}
\label{def:Vfunctor}
    For two $\V$-categories $\C$ and $\D$, a \define{$\V$-functor} $F \maps \C \to \D$ between these categories consists of a function $F \maps \Ob(\C) \to \Ob(\D)$ between the objects of $\C$ and $\D$ together with a morphism in $\V$
    \[
    F_{X,Y} \maps \C(X,Y) \to \D(F(X),F(Y)) \in \V
    \]
    for every pair $X,Y \in ob(\C)$, satisfying conditions of unit and composite preservation.
\end{defn}

\begin{defn}
\label{def:Vnattransf}
    Given two $\V$-functors $F,G \maps \C \to \D$, a \define{$\V$-natural transformation} $\alpha \maps F \To G$ consists of a morphism $\alpha_X \maps 0_\V \to \D(FX,GX)$ of $\V$ for each object $X \in ob(\C)$, satisfying a naturality condition.
\end{defn}

There is a 2-category $\VCat$ of $\V$-enriched categories, $\V$-functors, and $\V$-natural transformations.

\begin{rmk}
    In this paper, we always take $\V$ to be a closed symmetric monoidal \emph{poset}. In this case, $\V$ considered as an object of $\VCat$ is a posetal object. Indeed, every object in $\VCat$ is posetal, though the rest of this structure is not useful presently. 
\end{rmk}

\subsection{Stability in Lawvere Metric Spaces}
\label{sec:LawvereMetric}

Consider the monoidal poset $\Rplusinfty = [0,\infty]$, using the reverse $\geq$ of the usual order $\leq$, with monoidal product given by (extended) addition, $+$, and monoidal unit $0$. A category enriched in $\Rplusinfty$ is a \define{Lawvere metric space} \cite{LawvereMetric}. A Lawvere metric space is like an ordinary metric space in that its set of objects are taken to be the points of the space, and the function assigning pairs of objects to some element of $[0,\infty]$ obeys some of the properties of an ordinary metric. Indeed, the composition and identity morphism of \cref{def:enrichedcat} translate to the triangle inequality and the self-distance zero condition of ordinary metrics. 
Namely, letting $d(X,Y) = C(X,Y)$ be the \emph{Lawvere metric}, then given the monoidal structure $(\Rplus,+ ,0)$ with morphisms $\geq$ in $\Rplus$ yields: 
\begin{eqnarray*}
    \textrm{Identity morphism:}  && 0 \geq d(X,X) \\
    \textrm{Composition morphism:} && d(X,Y) + d(Y,Z) \geq d(X,Z). 
\end{eqnarray*}
In general, a Lawvere metric space is not symmetric, it does not necessarily obey $d(X,Y) = d(Y,X)$, and separability only holds up to isomorphism, $d(X,Y)=0$ implies $X \cong Y$.

The goal is to build a framework for settings for stability around Lawvere metric spaces. But doing to requires some care in properly defining the corresponding setting, i.e., one might be tempted to take the setting to be $\Rplusinfty\mhyphen\Cat$, the category of Lawvere metric spaces with $\Rplusinfty$-functors as morphisms. Yet this is overly restrictive. Namely, $\Rplusinfty$-functors are Lipschitz continuous functors with Lipschitz constant $1$; therefore, the morphisms in $\Rplusinfty\mhyphen\Cat$ are distance-shrinking functions. But first we note the following. 

\begin{rmk}
\label{rmk:productsinRplusCat}
    The category $\Rplusinfty\mhyphen\Cat$ has products because $\Rplusinfty$ has products (given by $\max$). Given to Lawvere metric spaces, $\C$ and $\D$, the product is given by taking the product in $\Set$ of object-sets $ob(\C\times\D) \coloneqq ob\C \times ob\D$, and the product in $\Rplusinfty$ of hom-objects, $\C\times \D((c_1,d_1),(c_2,d_2)) \coloneqq \max\{\C(c_1,c_2), \D(d_1,d_2)\}$.
\end{rmk}

\paragraph{\textbf{Limitations of $\Rplusinfty\mhyphen\Cat$.}}
To see why $\Rplusinfty\mhyphen\Cat$ is an overly restrictive setting for stability, take $T$ to either be $\Zplus$ or $\Rplus$ under addition. In either case, we realize them as metric spaces with the \define{discrete metric}, 
\[
    T(t_1,t_2) = \left\{ 
    \begin{array}{lcr}
    0 & \mathrm{if} & t_1 = t_2 \\
    \infty & \mathrm{if} & t_1 \neq t_2 
    \end{array} 
    \right. 
\]
A flow $\flow \maps T \times \C \to \C$ must be a distance-shrinking map. Since $T$ is assumed to be discrete, the metric on the product is only possibly finite when the $T$-component of the two objects are the same, in which case, the distance is precisely the distance between the $\C$-components.
\[
    T \times \C((t_1,c_1),(t_2,c_2)) = \left\{ 
    \begin{array}{lcr}
    \C(c_1,c_2) & \mathrm{if} & t_1 = t_2 \\
    \infty & \mathrm{if} & t_1 \neq t_2 
    \end{array} 
    \right. 
\]
Thus, $\flow$ is distance-shrinking if and only if $\flow_t \maps C \to C$ is distance-shrinking for all $t \in T$, $\C(c_1, c_2) \geq \C(\flow_t(c_1), \flow_t(c_2))$. 
An immediate consequence of this is that if the space has an equilibrium point $x^*$ with respect to a flow $\flow$, then any $T$-shaped trajectory $c \maps T \to \C$ must be monotonically approaching $x^*$, and hence $x^*$ is stable. The sorts of systems expressible in this category are precisely those which do not need Lyapunov theory to deduce stability. 

\paragraph{\textbf{The Category $\LMet$.}}
To address the limitations of $\Rplusinfty\mhyphen\Cat$ indicated above, define a new category: \define{Lawvere Metric Spaces, $\LMet$}, with objects being Lawvere metric spaces, and morphisms being arbitrary Lipschitz continuous functions. That is, a map in $\LMet$ is a function $f \maps X \to Y$ such that there exists an object $r \in \Rplus$ such that $f$ as a map of the form $r\cdot X \to Y$ is an $\Rplusinfty$-functor, where $r\cdot X$ is the metric space with $X$ as the underlying set, and metric given by $(x_1,x_2) \mapsto r \cdot X(x_1,x_2)$. Notice that $\Rplusinfty\mhyphen\Cat$ is a wide subcategory of $\LMet$ where $r$ can be chosen to be $1$. 
Additionally, the product in $\Rplusinfty\mhyphen\Cat$ (see Remark \ref{rmk:productsinRplusCat}) yields the product in $\LMet$. For a map $f \maps \Q \to \C$ with Lipschitz constant $K_f$, and a map $g \maps \Q \to \D$ with Lipschitz constant $K_g$, the function $(f,g) \maps Q \to \C \times \D$ has Lipschitz constant $K \coloneqq \max \{K_f, K_g\}$ because:
\begin{align*}
    \C\times \D((f,g)q_1,(f,g)q_2) 
    &= \max\{\C(fq_1,fq_2), \D(gq_1,gq_2)\}
    \\&\leq \max\{K_f,K_g\} \cdot \Q(q_1,q_2).
\end{align*}

\begin{example}
A \define{Lawvere Setting for Stability} is given by: 
\begin{enumerate}
\setcounter{enumi}{-1}
    \item[\textbf{S0:}] (setting) the category $\LMet$ of Lawvere metric spaces and Lipschitz continuous functions
    \item[\textbf{S1:}] (space) a Lawvere metric space $\C$
    \item[\textbf{S2:}] (time) A Lawvere metric space $\T$ equipped with the structure of a monoid $(\T, +, 0_{\T})$ with addition being Lipschitz continuous
    \item[\textbf{S3:}] (stable object) $\Rplusinfty$ itself viewed as a Lawvere metric space has:
    \begin{itemize}
        \item the monoidal unit $0 \in \Rplusinfty$ as the distinguished point 
        \item $\Rplusinfty$ has the structure of a posetal object of $\LMet$ (\cref{def:posetal}) as follows: for a metric space $\D$, the set $\LMet(\D, \Rplusinfty)$ of Lipschitz continuous functions has a partial order induced pointwise by the order on $\Rplusinfty$.
    \end{itemize}
    \item[\textbf{S4:}] (distance) the distance morphism is given by the enriched hom functor:
    \begin{eqnarray}
        d = \C(-,-) \maps \C\op \times \C  & \to &  \Rplusinfty \\
        (X,Y) & \mapsto &  \C(X,Y) \in \Rplusinfty \nonumber
    \end{eqnarray}
    which assigns to every pair of objects in $\C$ the hom-object in $\Rplusinfty$. 
\end{enumerate}
\end{example}

\begin{rmk}
    Notice that the type of the distance map here is slightly different than that of \cref{def:setting}. This is simply to account for the natural variance of a hom functor. It does not disallow the use of the results proven in \cref{sec:catlyap}, as all the proofs refer directly to the norm $d(x^*,\cdot)$ for a given equilibrium point, rather than the distance map $d$ itself. 
\end{rmk}

\paragraph{\textbf{Flows.}}  In $\LMet$, flows are permitted to not be strictly distance decreasing. Equilibrium points are as expected, and they are stable if there exists a Lipschitz continuous class $\K$ function $\alpha$ such that $\C(x^*, \flow(t,x)) \leq \alpha(\C(x^*,x))$ for all $t \in T$ and $x \in \C$. 

\paragraph{\textbf{Class $\K$ Morphisms.}}  Class $\K$ maps in $\LMet$ are order-preserving automorphisms of $\Rplusinfty$. In particular, for $\alpha \maps \Rplusinfty \to \Rplusinfty$ to be class $\K$, there must exist $r \in \Rplus$ such that $r(y-x) \geq \alpha(y) -\alpha(x)$. Note that not all automorphisms of $\Rplusinfty$ as a category are also $\Rplusinfty$-functors, nor are they necessarily Lipschitz continuous. 

\paragraph{\textbf{Lyapunov Morphisms.}}  A map $V \maps \C \to \Rplusinfty$ is then positive definite if there exist such class $\K$ maps $\underline \alpha, \overline \alpha \maps \Rplusinfty \to \Rplusinfty$ such that $\underline \alpha(\C(x^*,x)) \leq V(x) \leq \overline \alpha(\C(x^*,x))$ in $\Rplusinfty$. As in the classical case, this essentially says that 
\begin{itemize}
    \item $V(x) \geq 0$ for all $x \in \C$
    \item $V(x) = 0$ if and only if $x= x^*$
    \item as $\C(x^*,x) \to \infty$, $V(x) \to \infty$.
\end{itemize}

A Lyapunov morphism in this setting $V \maps \C \to \Rplusinfty$ is a Lipschitz continuous positive definite function with the property that $V(\flow(t,x)) \leq V(x)$ for all $x \in \C$. The generalized Lyapunov theorem \cref{thm:Lyap} says that if we have a Lyapunov morphism for an equilibrium point of a Lipschitz continuous flow on a metric space, then the point is stable.

\subsection{Stability in Enriched Categories}
\label{sec:StabEnrichedCat}

The salient properties of $\Rplusinfty$ that played a role in the previous section are the following:
\begin{itemize}
    \item the monoid structure on $\Rplusinfty$ given by addition allowed us to define a notion of category enriched in $\Rplusinfty$, giving our notion of space,
    \item the fact that $\Rplusinfty$ as a category is actually a poset means that all 2-cells are unique if they exist, justifying the use of lax commuting diagrams
    \item as a category, $\Rplusinfty$ has products, so $\Rplusinfty\mhyphen\Cat$ has products, so $\LMet$ has products
    \item $\Rplusinfty$ under addition is closed symmetric monoidal, and thus $\Rplusinfty$ gets to be self-enriched, and representable $\Rplusinfty$-enriched functors are a sensible notion.
\end{itemize}
It also must be pointed out that $\Rplusinfty$ being a poset means that enriched functors are really just a function on the objects which satisfies a property, and this enabled us to describe the generalized maps between metric spaces as a modification of this property. An analogous generalization for categories enriched in a general closed symmetric monoidal category which would work for stability would likely be much more involved. 

\paragraph{\textbf{Lipschitz $\V$-Functors.}} 
 Let $\V$ be a category equipped with two monoidal structures, denoted $(\oplus, 0)$ and $(\otimes, I)$, such that $A \otimes (B \oplus C) \cong (A \otimes B) \oplus (A \otimes C)$ and $0 \otimes A \cong A \otimes 0 \cong 0$ for all objects $A,B,C \in \V$. This is often called a \define{rig category} \cite{CoherenceDistributivity,CoherenceLaxAlgebras}, or a \define{bimonoidal category} \cite{SheetBimonoidal,JohnsonYauBimonoidal}. Consider categories enriched in the additive monoidal structure. For $r \in \V$ and $X \in \VCat$, let $rX$ denote the $\V$-category with object-set $ob(rX) \coloneqq ob(X)$ and hom-$\V$-object $rX(a,b) \coloneqq r \otimes X(a,b)$ as objects of $\V$. Composition is given by 
\begin{align*}
    rX(b,c) \oplus rX(a,b) 
    &= [r \otimes X(b,c)] \oplus [r \otimes X(a,b)]
    \\&\cong r \otimes [X(b,c) \oplus X(a,b)]
    \\&\xrightarrow{r \otimes \circ} r \otimes X(a,c)
    \\&= rX(a,c)
\end{align*}
and the identity is given by 
\[0 \cong r \otimes 0  \xrightarrow{r \otimes \iota} r \otimes X(a,a) =rX(a,a).\]
The associativity and unit laws that must hold between these depend on the coherence conditions that hold in bimonoidal categories.

% \joe{
% \begin{itemize}
%     \item why can't the constant be initial?
%     \item shouldn't we be asking one of the monoidal structures to be co/cartesian?
%     \item triple check that the identity structure morphism of $rX$ is sensible.
% \end{itemize}
% }

\begin{defn}
    Let $\V$ be a bimonoidal poset, and let $\C$ and $\D$ be $\V$-enriched categories. A \define{Lipschitz $\V$-functor} is a function $f \maps ob \C \to ob\D$ such that there exists a non-initial object $K_f \in \V$, termed a \define{Lipschitz object}, such that $f \maps K_f \C \to \D$ is a $\V$-functor. The composite of two Lipschitz functors $f \maps \C \to \D$ and $g \maps \D \to \E$ is Lipschitz with Lipschitz object $K_g \otimes K_f$. Let $\VCatgen$ denote the category of $\V$-categories and Lipschitz $\V$-functors.
\end{defn}

As the name suggests, when $\V = \Rplusinfty$, a Lipschitz $\V$-functor is exactly a Lipschitz continuous function. The reader should be careful to keep in mind that a Lipschitz $\V$-functor is a generalization of a $\V$-functor, rather than a special sort of $\V$-functor. Note that $\V$-functors are Lipschitz $\V$-functors with Lipschitz object $I$, the multiplicative unit of $\V$. Thus, $\VCat$ is a wide subcategory of $\VCatgen$, with $\VCat \hookrightarrow \VCatgen$ faithful. 

\begin{prop}
    If $\V$ has finite products, then so does $\VCatgen$. Additionally, the inclusion $\VCat \hookrightarrow \VCatgen$ preserves and creates products and limits. 
\end{prop}
\begin{proof}
    It is well-known that if $\V$ has finite products, then so does $\VCat$, and the product $\C \times \D$ has object set $ob\C \times ob\D$ and hom-objects \[C\times D((c_1,d_1),(c_2,d_2)) = \C(c_1,c_2) \times \D(d_1, d_2).\] We claim that this gives the product in $\VCatgen$ as well. Notice that the product on the right hand side of the equation above is the product in $\V$. Given our convention of using the order $\geq$ for posets, this product is generally denoted by $\max$. 

    Let $\E$ be a $\V$-category, and $p \maps \E \to \C$ and $q \maps \E \to \D$ Lipschitz functors. The underlying functions on objects admit a pairing $(p,q) \maps ob\E \to ob\C \times ob\D$. We can see this is a Lipschitz $\V$-functor with Lipschitz object $\max\{K_p,K_q\}$ because 
    \begin{align*}
        \C\times \D((p,q)(e_1),(p,q)(e_2))
        &= \max\{\C(p(e_2),p(e_2)), \D(q(e_1),q(e_2))\}
        \\&\leq \max\{K_p,K_q\} \E(e_1,e_2)
    \end{align*}
    for $e_i \in \E$. 

    Similarly, the terminal object $1$ of $\VCatgen$ has a single object, and the unique hom-object is the terminal object $1_\V$ of $\V$. For any $\V$-category $X$, there is a unique function of the form $! \maps ob X \to ob1$. For any pair of objects $a,b \in X$, there is a unique map $X(a,b) \to 1(\ast,\ast) = 1_\V$. Thus this is a $\V$-functor. 
\end{proof}

We abstract the properties of $\Rplusinfty$ listed above to discuss settings for stability in enriched categories. 

\begin{defn}
Let $\V$ be a closed symmetric bimonoidal poset with products. Then a \define{$\V$-enriched setting for stability} in $\VCatgen$, is given by: 
\begin{enumerate}
\setcounter{enumi}{-1}
    \item[\textbf{S0:}] (setting) the category $\VCatgen$ of $\V$-categories and Lipschitz $\V$-functors
    \item[\textbf{S1:}] (space) a $\V$-category $\C$
    \item[\textbf{S2:}] (time) A strict monoidal $\V$-category $(\T, +, 0_{\T})$.
    \item[\textbf{S3:}] (stable object) the enriching category $\V$ itself viewed as a $\V$-category has:
    \begin{itemize}
        \item the distinguished point $0_{\V} \in \V$ 
        \item $\V$ has the structure of a posetal object of $\VCatgen$ (\cref{def:posetal}) as follows: for a $\V$-category $\D$, the set $\VCatgen(\D, \V)$ of Lipschitz functors inherits a partial order pointwise from $\V$.
    \end{itemize}
    \item[\textbf{S4:}] (distance) the distance morphism is given by the hom functor:
    \begin{eqnarray}
    \label{eqn:distanceVfunctor}
        d = \C(-,-) \maps \C\op \times \C  & \to &  \V \\
        (X,Y) & \mapsto &  \C(X,Y) \in \V \nonumber
    \end{eqnarray}
    which assigns to every pair of objects in $\C$ the hom-object in $\V$. 
\end{enumerate}
\end{defn}

\begin{lem}
\label{lem:distance}
For a $\V$-enriched setting for stability in $\VCatgen$, the distance $\V$-functor \eqref{eqn:distanceVfunctor} satisfies for all $X,Y,Z \in \C$:
\begin{enumerate}[(i)]
    \item[(i)] Positivity: $d(X,Y) \geq 0_{\V}$, 
    \item[(ii)] Separability: $d(X,Y) = 0$ and $d(Y,X) = 0$ if and only if $X \cong Y$. 
    \item[(iii)] Triangle Inequality: $d(X,Y) \otimes d(Y,Z) \geq d(X,Z)$. 
\end{enumerate}
\end{lem}

\begin{proof}
(i) \emph{Positivity:} follows from the assumption that $0_{\V}$ is terminal. 

(ii) \emph{Separability:}  By definition, two objects $X, Y$ of a $\V$-category are isomorphic if there are maps of $\V$ of the form $f \maps 0_\V \to C(X,Y)$ and $g \maps 0_\V \to C(Y,X)$ which make the following squares commute.
\[
\begin{tikzcd}
    0_\V
    \arrow[r, "\sim"]
    \arrow[d, "id_a"']
    & 
    0_\V \otimes 0_\V
    \arrow[d, "g \otimes f"]
    \\
    C(a,a)
    &
    C(b,a) \otimes C(a,b)
    \arrow[l, "\circ"]
\end{tikzcd}
\quad
\begin{tikzcd}
    0_\V
    \arrow[r, "\sim"]
    \arrow[d, "id_b"']
    & 
    0_\V \otimes 0_\V
    \arrow[d, "f \otimes g"]
    \\
    C(b,b)
    &
    C(a,b) \otimes C(b,a)
    \arrow[l, "\circ"]
\end{tikzcd}
\]
Since $\V$ is a poset, the existence of such maps means that $\C(X,Y) \leq 0_\V$ and $\C(Y,X) \leq 0_\V$, and the two squares above commute automatically. Thus (ii) follows from (i).

(iii) \emph{Triangle Inequality:} follows from the composition morphism. 
\end{proof}

\begin{rmk}
\label{rem:symmetric}
    As noted in Section \ref{sec:LawvereMetric}, Lawvere metrics are not necessarily symmetric and, as a result, the distance morphism associated with a $\V$-enriched setting for stability are not necessarily symmetric. A Lawvere metric space $\C$ can be symmetrized, resulting in a new metric space $sym\C$ with the same objects and metric defined by 
    \[sym\C(X,Y) = \max\{ \C(X,Y),\C(Y,X) \}.\]
    Similarly, if $\V$ is a symmetric monoidal poset, any $\V$-category can be symmetrized by the same formula.
    
    In these cases, the result is the following additional property on the distance $\V$-functor \eqref{eqn:distanceVfunctor}: for $X, Y \in \C$: 
    \begin{itemize}
        \item[(iv)] Symmetry: $d(X,Y) = d(Y,X)$. 
    \end{itemize}
\end{rmk}

\begin{rmk}
If $\V$ additionally has an initial object, which we choose to denote by $\infty$ in analogy with $\Rplusinfty$, then any set $X$ can be equipped with a \define{discrete $\V$-category} structure. For an element $x \in X$, the hom-object $\disc X(x,x)$ is defined to be the terminal object $0_\V$, and for any other object $y \neq x$, the hom-object $\disc X(x,y)$ is defined to be the initial object $\infty$. This extends to a functor $\disc \maps \Set \to \VCatgen$. 

The functor $\disc$ is fully faithful as functions out of a discrete $\V$-category are automatically $\V$-functors, and preserves products as the product of discrete $\V$-categories is again discrete. Thus, as in the case of metric spaces, both of the monoids $(\Zplus, +,0)$ and $(\Rplus, +,0)$ equipped with the discrete $\V$-category structure are monoid objects in $\VCat$, and thus in $\VCatgen$ as well. 
\end{rmk}

\paragraph{\textbf{Flows.}} Let $T$ be a discrete monoid object in $\VCatgen$. A flow $\flow \maps T \times \C \to \C$ is then a monoid action of $T$ on $ob \C$ which is a Lipschitz functor. Let $K_\flow$ be a Lipschitz object for $\flow$. Since $T$ is discrete, the property of $\flow$ being a Lipschitz functor is equivalent to the property: for all $t \in T$, the map $\flow(t, -) \maps \C \to \C$ is a Lipschitz functor with Lipschitz object $K_\flow$. Equilibrium points are as expected, and they are stable if there exists a Lipschitz class $\K$ functor $\alpha$ such that $\C(x^*, \flow(t,x)) \leq \alpha(\C(x^*,x))$ for all $t \in T$ and $x \in \C$. 

\paragraph{\textbf{Class $\K$ Morphisms.}} A class $\K$ map in an enriched setting is an order-preserving invertible Lipschitz functor $\alpha \maps \V \to \V$ which preserves the monoidal unit $0_\V$. In particular, there exists an object $K_\alpha \in \V$ such that  $K_\alpha \otimes \V(x,y) \geq \V(\alpha x, \alpha y)$ for all $x,y \in \C$. Not all automorphisms of $\V$ as a poset are also Lipschitz functors. 

\paragraph{\textbf{Lyapunov Functors.}} A Lipschitz functor $V \maps \C \to \V$ is then positive definite if there exist such class $\K$ maps $\underline \alpha, \overline \alpha \maps \V \to \V$ such that $\underline \alpha(\C(x^*,x)) \leq V(x) \leq \overline \alpha(\C(x^*,x))$ in $\V$. As in the case of Lawvere metric spaces, this implies $V$ is \define{positive definite} if:
\begin{itemize}
    \item $V(x) \geq 0_\V$ for all $x \in \C$
    \item $V(x) = 0_\V$ if and only if $x = x^*$
    \item as $\C(x^*,x) \to \infty$, $V(x) \to \infty$.
\end{itemize}
The categorical Lyapunov theorem, \cref{thm:Lyap}, therefore yields the following:

\begin{cor}
    Let be a $\V$-enriched setting for stability in $\VCatgen$. Then $V \maps \C \to \V$ is a Lyapunov morphism if it is positive definite and $V(\flow(t,x)) \leq V(x)$ for all $x \in \C$ for a Lipschitz flow $\flow$ with equilibrium point $x^*$. The existence of a Lyapunov morphism in $\VCatgen$ implies the stability of $x^*$. 
\end{cor}

\subsection{Set-Based Stability}
\label{sec:powersetenriched}

The advantage of the $\V$-enriched setting for stability is that it supplies an inherent notion of distance via the hom functor. This is obvious in the metric space case where the hom functor is literally a metric. In this section, we demonstrate the point further with an example that produces a non-trivial notion of distance. 

For a set $E \in \Set$, the power set $\P(E)$ is a complete lattice. In line with the previous sections, we take the order to be $\supseteq$, the reverse of the usual order when thinking of $\P(E)$ as a category. In this category, the product is given by union $\cup$, the terminal object is the empty set, internal hom $[U,V]$ is given by set difference $V \setminus U$. 

Indeed, for any $U',U,V \subseteq E$ we have the isomorphism: 
\begin{eqnarray}
\P(E)(U' \cup U , V) & \cong & \P(E)(U',[U,V]) \nonumber\\
(U'\cup U)\supseteq V  
& \leftrightarrow  &  
U'\supseteq (V\setminus U). \nonumber
\end{eqnarray}

Being closed, the category $\P(E)$ is self-enriched, with $\P(E)(U, V) = V \setminus U$. Yet this is not symmetric. We can instead consider the symmetrization $\P(E)^{sym}$:
\begin{itemize}
    \item symmetric hom object: $\P(E)^{sym}(U,V) \coloneqq (U \setminus V) \cup (V \setminus U) \in \P(E)$.
    \item composition morphism: 
    \[
    \underbrace{ (U \setminus V) \cup ( V \setminus U ) }_{\P(E)^{sym}(U,V)}\cup  \underbrace{(V \setminus W) \cup (W \setminus V)}_{\P(E)^{sym}(V,W)} \supseteq \underbrace{(U \setminus W) \cup (W \setminus U)}_{\P(E)^{sym}(U,W)}. 
    \]
    \item identity morphism: $\emptyset \supseteq \P(E)^{sym}(U,U) = (U \setminus U) \cup (U \setminus U) = \emptyset$ trivially.
\end{itemize}  
Consider the $\P(E)$-enriched setting for stability in $\P(E)\mhyphen\Cat$: 
\begin{itemize}
    \item time object $T$
    \item object of interest $\P(E)^{sym}$
    \item measurement object $\P(E)$
\end{itemize}
Here, per Definition \ref{def:setting}, $\T$ is a discrete monoidal $\P(E)$-category. The stable object is simply $\P(E)$ as a self-enriched category, which is a valid stable object as $0_{\P(E)} = \emptyset$ is terminal, and $\P(E)$ is a poset. At last, we note that the distance is given by the (symmetric) hom functor: 
\begin{eqnarray}
\label{eqn:distanceVfunctorPE}
    d \maps \P(E) \times \P(E)  & \to &  \P(E) \\
    (U,V)) & \mapsto &  (U \setminus V) \cup (V \setminus U)  \nonumber
\end{eqnarray}
This can be used to define set-based stability. 

Consider a set $E$ (not necessarily a subset of $\mathbb{R}^n$). Consider a flow on this set $\flow \maps \mathbb{R}_{\geq 0} \times E  \to  E$. For $\P(E)$ the corresponding power set, this induces a flow: 
\begin{eqnarray}
\label{eqn:setflow}
    \flow \maps \mathbb{R}_{\geq 0} \times \P(E) & \to & \P(E) \\
    (t,U) & \mapsto & \flow_t(U) \coloneqq \{\flow_t(x) \mid x \in U\}  \nonumber
\end{eqnarray}
We can study the stability of this system through a modified setting for stability defined by: 
\begin{enumerate}
\setcounter{enumi}{-1}
    \item[\textbf{S0:}] (setting) the category $\Set$, with terminal object $1 = \{*\}$. 
    \item[\textbf{S1:}] (space) The space of interest is the power set $\P(E)$ of a set $E$. 
    \item[\textbf{S2:}] (time) the object $\mathbb{R}_{\geq 0}$ is a monoid with $+$ addition and $0_\oplus = 0$. 
    \item[\textbf{S3:}] (stable object) the power set $\P(E)$. This has a monoidal structure: $(\P(E), \cup, \emptyset)$, wherein: 
    \begin{itemize}
        \item distinguished point: $\emptyset \in \P(E)$. Therefore, the zero morphism, $0$, is given by the constant empty set function $0(U) = \emptyset$. 
        \item posetal structure: $f \To g$, if $f(U) \supseteq g(U)$ for all $U \in \P(E)$. 
    \end{itemize} 
    \item[\textbf{S4:}] (distance) the morphism:
    \begin{eqnarray}
        d \maps \P(E) \times \P(E)  & \to &  \P(E) \\
        (U,V) & \mapsto & (U \setminus V) \cup (V \setminus U )   \nonumber
    \end{eqnarray}
    where $(U \setminus V) \coloneqq \{ u \in U \mid u \notin V \}$. This is simply the (symmetric) hom functor obtained from the enrichment of $\P(E)$ over itself.
\end{enumerate}

\paragraph{\textbf{Norm.}}  The distance yields yields the (categorical) norm: $\|\cdot\|_{\{x^*\}} \maps \P(E) \to R$ defined to be:
\begin{align*}
    \|U\|_{\{x^*\}} 
    & \coloneqq (U \setminus \{x^*\}) \cup (\{x^*\} \setminus U ) 
    \\&= \{ u \in U \mid u \neq x^* \} \cup \{ x^* \mid x^* \notin U \}  \nonumber
    \\&= \left\{ 
    \begin{array}{ccl}
        U \setminus \{x^*\}  & \mathrm{if} & x^* \in U \\
        \{x^*\} & \mathrm{if} & x^* \notin U 
    \end{array}
    \right.
\end{align*}

It naturally satisfies: 
\begin{itemize}
    \item $ \|\cdot\|_{\{x^*\}} \To 0$ since $\| U \| \supseteq \emptyset$,
    \item  $\ker(\|\cdot\|_{\{x^*\}}) = \{*\}$ since $\| U \|_{\{x^*\}} = \emptyset$ iff  $U = \{x^*\}$
\end{itemize}  
        
\paragraph{\textbf{Class $\K$ Morphisms.}}  A class $K$ morphism $\alpha: \P(E) \to \P(e)$ satisfies: 
\begin{itemize}
    \item $\alpha$ is strictly increasing on $\P(E)$: $U_1 \subset U_2$ implies that $\alpha(U_1) \subset \alpha(U_2)$, 
    \item $\alpha$ maps $\emptyset \in \P(E)$ to $\emptyset$: $\alpha(\emptyset) = \emptyset$. 
\end{itemize}
A map $V \maps \P(E) \to \P(E)$ is:
\begin{itemize}
    \item \define{positive semi-definite} if $V(U) \supseteq \emptyset$, 
    \item \define{positive definite} if in addition $V(U) = \emptyset$ iff $U = \{x^*\}$. 
\end{itemize}
The map $V$ is \textbf{class $\K$ bounded} if there exists class $\K$ maps $\underline{\alpha}, \overline{\alpha}$ such that:  
\[
\underline{\alpha}(\| x \|_{\{x^*\}} ) \subseteq V(U) \subseteq 
\overline{\alpha}( \| x \|_{\{x^*\}} )
\]
wherein if $x^* \notin U$, $V(U) \cong \{x^*\}$ and if $x^* \in U$: 
\[
    \underline{\alpha}(U \setminus \{x^*\})
    \subseteq V(U) \subseteq 
    \overline{\alpha}(U \setminus \{x^*\})
\]

\paragraph{\textbf{Flows.}} Flows are defined as in \eqref{eqn:setflow}. 

\paragraph{\textbf{Stability.}}  An \textbf{equilibrium point} satisfies $\flow_t(\{x^*\}) = \{x^*\}$. This equilibrium point is stable if: 
\[
    \| \flow_t(U) \|_{\{x^*\}} \subseteq \alpha(\| U \|_{\{x^*\}})
\]
for a class $\K$ map $\alpha$. This can be interpreted classically as follows: 
\begin{lem}
     If $\{x^*\}$ is stable, then for all $V \subset E$ with $x^* \in V$, there exists a $U \subset E$ with $x^* \in U$ such that: 
     \[
    \flow_t(U) \subseteq V, \quad \forall t \in \Rplus
     \]
     If $E \subset \mathbb{R}^n$ then for all $\epsilon > 0$ there exists a $\delta > 0$ such that: $\flow_t(B_{\delta}(x^*)) \subseteq B_{\epsilon}(x^*)$. 
\end{lem}

\begin{proof}
Given $V \subset E$, pick
% \footnote{If $x^* \notin U$, then pick  $U = \alpha^{-1}(V) \cup \{x^*\}$.}  
$U = \alpha^{-1}(V)$. Then: 
\begin{align*}
\flow_t(U)
= 
\| \flow_t(U) \|_{\{x^*\}} \cup \{x^*\}
& \subseteq \alpha(\| U \|_{\{x^*\}}) \cup \{x^*\}
= 
\alpha(U \setminus \{x^*\}) \cup \{x^*\} \nonumber\\
& \subseteq V  \setminus \{\alpha(x^*)\} \cup \{x^*\} \subseteq V 
\end{align*}
In the case when  $E \subset \mathbb{R}^n$, given $\epsilon > 0$, let $V = B_{\epsilon}(x^*)$ and $U = \alpha^{-1}(B_{\epsilon}(x^*))$ and pick $\delta > 0$ such that $B_{\delta}(x^*) \subseteq U$. The result follows from the fact that $\alpha(B_{\delta}(x^*)) \subseteq \alpha(U)$. 
\end{proof}

\paragraph{\textbf{Lyapunov's Theorem.}}  A Lyapunov morphism in this case satisfies: 
\begin{enumerate}
    \item $\underline{\alpha}(\| x \|_{\{x^*\}} ) \subseteq V(U) \subseteq 
\overline{\alpha}( \| x \|_{\{x^*\}} )$
    \item $V(\flow_t(U)) \subseteq V(U)$. 
\end{enumerate}
Theorem \ref{thm:Lyap} implies that, in that case, $x^*$ is a stable equilibrium point. 

\begin{rmk}
    The set-based version of Lyapunov's  theorem was proven categorically in \cite{MattenetJungers}, where necessary and sufficient conditions were given. In that work, condition (1) is represented by requiring $V$ be a \emph{rough approximation} of $\|x\|_{x^*}$, and condition (2) is represented by requiring $V$ to be \emph{monovariant}. 
\end{rmk}

\subsection*{Acknowledgments}

This research was supported by the Air Force Office of Scientific Research under the Multidisciplinary University Research Initiative grant Hybrid Dynamics ‐ Deconstruction and Aggregation (HyDDRA). First and foremost, we would like to thank  S\'ebastien Mattenet for his collaboration in formulating the converse Lyapunov conditions in Theorem \ref{thm:Lyapconv}, building upon his previous work \cite{MattenetJungers}. This collaboration lead to the follow-on paper extending this work to systems \cite{CLT2}. We would also like to thank David Spivak for the numerous discussions on the categorical formulation of Lyapunov theory, and specifically the framing of Proposition \ref{prop:lyapunov}. The first author would like to thank Mariusz Wodzicki for the introduction to Homotopical Algebra. Finally, we would like to thank John Baez and Todd Trimble for the many discussions on category theory.

\bibliographystyle{plain}
\bibliography{references}

\end{document}